\input xy \xyoption{all}
\documentclass[12pt,leqno,a4paper,british]{amsart}

\usepackage{amscd,amssymb,amsmath,latexsym}
\usepackage{babel}

\usepackage{url,enumerate}

\usepackage[
bookmarks=true,
breaklinks=true,
bookmarksnumbered = true,
colorlinks= true,
urlcolor= green,
anchorcolor = yellow,
citecolor=blue,
]{hyperref}                   

\newcommand*{\lhra}{\ensuremath{\lhook\joinrel\relbar\joinrel\rightarrow}}

\title{\sc On the homotopy fibre of the inclusion map $F_n(X)\lhra\prod_1^nX$ for some orbit spaces $X$}

\author{Marek Golasi\'nski, Daciberg Lima Gon\c{c}alves and John Guaschi}

\date{\today}

\def\R{{\mathbb R}}
\def\S{{\mathbb S}}
\def\Z{{\mathbb Z}}

\newtheorem{thmX}{Theorem}[section]
\newtheorem{proX}[thmX]{Proposition}
\newtheorem{lemX}[thmX]{Lemma}
\newtheorem{CorX}[thmX]{Corollary}

\newtheoremstyle{remarkk}{}{}{}{}{\scshape}{.}{ }{}

\theoremstyle{remarkk}
\newtheorem{RemX}[thmX]{Remark}

\begin{document}

\vspace*{3cm}

\begin{abstract}  \emph{Under certain conditions, we describe the homotopy type of the homotopy fibre of the inclusion map
$F_n(X)\lhra\prod_1^nX$ for the $n^{th}$ configuration space $F_n(X)$  of a topological manifold $X$ without boundary such that
$\operatorname{\text{dim}}(X)\ge 3$. We then apply our results to the cases where either the universal covering of $X$ is contractible
or $X$ is an orbit space $\mathbb{S}^k/G$ of a tame, free action of a Lie group $G$ on the  $k$-sphere $\mathbb{S}^k$.
If the group $G$ is finite and $k$ is odd, we give a full description of  the long exact sequence in homotopy of the homotopy fibration
of the inclusion map $F_n(\mathbb{S}^k/G)\lhra\prod_1^n\mathbb{S}^k/G$.}
\end{abstract}

\subjclass[2010]{Primary: 55R80; secondary: 20F36, 20J06}
\keywords{configuration space, free action of a group, homotopy fibre, homotopy pull-back, Lie group, manifold, sphere, orbit space, whisker map}

\maketitle

\noindent
\emph{This paper is dedicated to Professor S.~Gitler, for his outstanding work in homotopy theory, his readiness to encourage  the scientific progress of colleagues and of institutions, his unwavering intellectual honesty, and his warm friendship.}

\bigskip

\noindent
{\large\bf Introduction}

\medskip

Configuration spaces have a long history and continue to be an active topic of research, due to their interesting topological and geometric
properties~\cite{Coh,FH}. They arise in various areas in mathematics, such as low-dimensional topology, homotopy theory, dynamical systems
and hyperplane arrangements, and play an important r\^{o}le in the study of braid groups~\cite{Bi2}.
If $n\geq 1$, the \emph{$n$\textsuperscript{th} configuration space} of a topological space $X$ is defined to be the subspace of the $n$-fold Cartesian product $\prod_1^nX$ defined by:
\begin{equation}\label{eq:defconfig}
 F_n(X)=\Bigl.\Bigl\{(x_1,\ldots,x_n)\in\prod_1^nX \;\Bigr\rvert \; \text{$x_i\ne x_j$ for all $i\ne j$} \Bigr\}.
\end{equation}
The classical notion of configuration space has been generalised in several directions,
one of them being the following. Given a free action $G\times X\longrightarrow X$ of a group $G$ on $X$, we define
the {\em $n$\textsuperscript{th} orbit configuration space} of $X$ (with respect to $G$) as follows~\cite{CX}:
\begin{equation}\label{eq:deforbconfig}
F_n^G(X)=\biggl\{ \biggl. (x_1,\ldots,x_n)\in \prod_1^nX\,\biggr\rvert\, \text{$Gx_i\cap Gx_j=\emptyset$ for all $i\neq j$}\biggr\}.
\end{equation}
Note that if $G$ is the trivial group then $F_{n}^{G}(X)$ coincides with the usual configuration space $F_{n}(X)$.

One general question that is related to the study of $F_{n}(X)$ is the analysis of the inclusion map $i_{n}(X): F_n(X) \lhra \prod_1^n X$, and if one is interested in the homotopy groups of these spaces, the homotopy fibre $I_{i_{n}(X)}$ of this map (up to homotopy equivalence). The study of the homotopy fibre of this inclusion map has been  carried out for certain manifolds $X$. If $X$ is a one-dimensional manifold (the circle $\mathbb{S}^1$ or the real line $\mathbb{R}^1$) then by~\cite{We}, the configuration space $F_n(\mathbb{R}^1)$ is homeomorphic to a disjoint union of $n!$ copies of the open $n$-simplex, and using~\cite[Example 2.6]{Coh}, there is a homeomorphism $\mathbb{S}^1\times F_n(\mathbb{R}^1)\stackrel{\approx}{\longrightarrow} F_{n+1}(\mathbb{S}^1)$. One may then describe the homotopy fibre of the inclusion map $i_{n}(X): F_n(X)\lhra\prod_1^nX$ for $X=\mathbb{R}^1$ or $\mathbb{S}^1$.

The study of configuration spaces for surfaces was initiated in~\cite{FaN}. In the case, the question about the homotopy fibre of $i_{n}(X)$ was considered by Goldberg in~\cite{Gol} when the surface is different from $\mathbb{S}^2$ and $\mathbb{R}P^2$, in which case the homotopy fibre is aspherical.  Recently the cases when $X$ is either $\mathbb{S}^2$ or $\mathbb{R}P^2$ have been studied in~\cite{GG10,GG11}, and it was shown that the homotopy fibre of $i_{n}(X)$ is no longer aspherical.

In this paper, we study the homotopy fibre of the inclusion map $i_{n}(X): F_n(X)\lhra\prod_1^nX$ for  topological manifolds $X$ without boundary and of dimension $\geq 3$. As special cases, if the universal covering of $X$ is contractible or $X$ is an orbit space $\mathbb{S}^k/G$ of a free, tame action of a Lie group $G$ on the $k$-sphere $\mathbb{S}^k$ then we describe the homotopy type of the homotopy fibre $I_{i_{n}(X)}$ of the map $i_{n}(X)$. In particular if $X=\mathbb{S}^k/G$, the group $G$ is finite and $k\ge 3$ is odd, we give a full description of the long exact sequence in homotopy of the homotopy fibration of the inclusion map $i_{n}(X): F_n(X)\lhra\prod_1^nX$.

Let $X$ and $Y$ be topological spaces. We write $X\simeq Y$ if $X$ and $Y$ have the same homotopy type, and 
if $f,g : X\longrightarrow Y$ are maps, we write $f\sim g$ if they are homotopic. The \emph{connectivity} of $X$, denoted by $\operatorname{\text{conn}}(X)$, is defined to be the largest non-negative integer $n$ (or infinity) for which $\pi_{i}(X)=0$ for all $i\leq n$. In Section~1, we recall and prove various results about pull-backs with a homotopical approach to fibrations that will be used in the rest of the paper. It is well known that a homotopy-commutative diagram of the form:
$$\xymatrix{X\ar[d]_g\ar[rr]^f && Y\ar[d]^{g'}\\
X'\ar[rr]^{f'}&&Y'}$$
gives rise to a map $\mathcal{I}(g,g'): I_{f} \longrightarrow I_{f'}$ between the homotopy fibres of $f$ and $f'$ that is compatible with this diagram (see Remark~\ref{defIgg} for the precise definition of this map and some of its properties). Proposition~\ref{LL} highlights various homotopy fibrations and homotopy equivalences that arise from such homotopy-commutative diagrams, and is main result of Section~1.

In Section~2, $X$ will mostly be a topological manifold without boundary such that $\dim(X)\geq 3$, and we shall frequently assume that there exists an action of a topological group $G$ on $X$. At certain points, we will suppose additionally that the action is tame~\cite{edmonds}, so that the quotient space $X/G$ is also a topological manifold. In this case, if $Q_{1}\in X$, the map $i_{n}(X): F_n(X)\lhra\prod_1^nX$ and the quotient map $q(X): X \longrightarrow X/G$ induce a map $\psi_{n}(X): F_{n}^{G}(X) \longrightarrow F_{n}(X/G)$ between the $n$\textsuperscript{th} orbit configuration space of $X$ and the $n$\textsuperscript{th} configuration space of $X/G$, an inclusion $i_{n}''(X): F_{n}^{G}(X \backslash GQ_{1})  \lhra\prod_1^nX$, and an inclusion $i_{n}''(X): F_{n}^{G}((X/G)\backslash \bar{Q}_{1})  \lhra\prod_1^nX/G$, where $\bar{Q}_{1}=q(X)(Q_{1})$. These maps are defined at the beginning of Section~2. We also define maps $j_{n}(X): \prod_1^{n-1}X \longrightarrow \prod_1^{n}X$ and $j_{n}'(X): F_{n-1}(X\backslash Q_{1}) \longrightarrow F_{n}(X)$ that `insert' the point $Q_{1}$ in a given position (see equation~(\ref{E'})).

In the first part of Section~2, we shall study these maps, the aim being to relate their homotopy fibres to that of the map $i_{n}(X): F_n(X)\lhra\prod_1^nX$, and to obtain (weak) homotopy equivalences between them. In
Theorem~\ref{B}, we prove a result concerning the connectivity of the map $i_{n}(X)$ that extends a result of Birman from the smooth category to the topological category~\cite[Theorem~1]{Bi1}. The following theorem, which is the main result of Section~2, describes certain properties of the long exact sequence in homotopy of the fibration associated with the inclusion map $F_n(X)\lhra\prod_1^nX$.

\medskip

\noindent
{\bf Theorem \ref{l}.} {\em Let $G\times X\longrightarrow X$ be a free, tame action of a Lie group $G$ on a connected topological manifold $X$ without boundary, and
let $Q_{1}\in X$. Suppose that the inclusion map $i''_{n-1}(X): F_{n-1}^{G}(X\backslash GQ_1)\lhra\prod_1^{n-1}X$ is null homotopic
(this is the case if for example the inclusion map $X\backslash GQ_1\lhra X$ is null homotopic).
\begin{enumerate}[(1)]
\item The maps $\mathcal{I}(\mbox{\em id}_{F_{n-1}^G(X\backslash GQ_1)}, \mbox{\em id}_{\prod_1^{n-1}X}): I_{i_{n-1}''(X)}
\longrightarrow F_{n-1}^G(X\backslash GQ_1) \times \Omega(\prod_1^{n-1}X)$ and $\mathcal{I}(j_{n'}(X/G),j_{n}(X/G)): I_{i'_{n-1}(X/G)}\longrightarrow
I_{i_n(X/G)}$ are homotopy equivalences, and the map $\mathcal{I}(\psi_{n-1}(X \backslash GQ_1), \prod_1^{n-1}q(X)): I_{i_{n-1}''(X)}\longrightarrow
I_{i'_{n-1}(X/G)}$ is a weak homotopy equivalence.  Furthermore, if $G$ is discrete, $X$ is
$1$-connected and $\dim(X/G)\ge 3$ then the map $\psi_{n-1}(X\backslash GQ_1): F^G_{n-1}(X\backslash GQ_1)\longrightarrow F_{n-1}((X/G)\backslash \bar Q_1)$ is the universal covering.
\item If $j\leq \min(\operatorname{conn}(X), \dim (X/G)-2)$ then $\pi_j(F_{n-1}^G(X\backslash GQ_1))=0$.
\item\label{it:dummy} Suppose that the homomorphism $\pi_j(X)\longrightarrow\pi_j(X/G)$ induced by the quotient map $X\longrightarrow X/G$ is injective
for all $j\geq 1$ (this is the case if for example the inclusion map $GQ_1\lhra X$ is null homotopic).
For all $j\geq 2$, up to the identification of the groups $\pi_{j-1}(I_{i_n(X/G)})$ and $\pi_{j-1}(F^G_{n-1}(X\backslash GQ_1))\times\pi_{j-1}(\Omega(\prod_1^{n-1}X))$ via the isomorphism
\begin{equation*}
\textstyle\pi_{j-1}(\mathcal{I}(\mbox{\em id}_{F_{n-1}^G(X\backslash GQ_1)}, \mbox{\em id}_{\prod_1^{n-1}X})) \circ (\pi_{j-1}(k_{n}(X)))^{-1},
\end{equation*}
the restriction to the subgroup $\pi_{j}(\prod_1^{n-1}X)$ of $\pi_{j}(\prod_1^{n}X/G)$ of the boundary homomorphism $\hat\partial_{j}: \pi_{j-1}(\Omega(\prod_1^nX/G)) \longrightarrow \pi_{j-1}(I_{i_n(X/G)})$ given by the long exact sequence in homotopy of the homotopy fibration $I_{i_n(X/G)}\xrightarrow{p_{i_n(X/G)}} F_n(X/G)\stackrel{i_n(X/G)}{\lhra} \prod_1^nX/G$
of $(\ref{threefibres})$ coincides with the inclusion of $\pi_{j}(\prod_1^{n-1}X)$ in
$\pi_{j-1}(F^G_{n-1}(X\backslash GQ_1))\times\pi_{j-1}(\Omega(\prod_1^{n-1}X))$ via the usual identification of $\pi_{j}(X)$ with
$\pi_{j-1}(\Omega(X))$.
\end{enumerate}}


The map $k_{n}(X): I_{i_{n-1}''(X)} \longrightarrow I_{i_n(X/G)}$ mentioned in Theorem~\ref{l}(\ref{it:dummy}) is a weak homotopy equivalence that is defined just after the commutative diagram~(\ref{threefibres}). By applying Theorem~\ref{l} to manifolds whose universal covering is contractible, in Proposition~\ref{AS} we show  that the homotopy fibre $I_{i_n(X/G)}$ is weakly homotopy equivalent to a number of different spaces that arise in the construction of the above-mentioned maps.

In Section~3, we apply the results of Section~2 to orbit manifolds of the form $\mathbb{S}^k/G$, $G$ being a finite group that acts freely and tamely on the $k$-sphere $\mathbb{S}^k$, where $k\geq 3$. One of the two principal results of Section~3 is the following corollary that describes the homotopy type of the homotopy fibre of the inclusion $i''_{n-1}(\mathbb{S}^k): F_n^{G}(\mathbb{S}^k\backslash GQ_{1})\lhra\prod_1^{n-1}\mathbb{S}^k$.

\medskip

\noindent
{\bf Corollary \ref{P}.} {\em Let $G\times \mathbb{S}^k\longrightarrow \mathbb{S}^k$ be a free, tame action of a compact Lie group $G$ on $\mathbb{S}^k$, and let $Q_{1}\in \mathbb{S}^k$.
\begin{enumerate}[(1)]
\item There is a homotopy equivalence $I_{i''_{n-1}(\mathbb{S}^k)} \simeq F_{n-1}^G(\mathbb{S}^k\backslash GQ_1)\times \Omega(\prod_1^{n-1}\mathbb{S}^k)$.
If  the group $G$ is finite and $\dim(\mathbb{S}^k/G)\geq 3$ then $\widetilde{F_{n-1}(\mathbb{S}^k\backslash Q_1)} \simeq F_{n-1}^G(\mathbb{S}^k\backslash GQ_1)$.

\item If $j\leq k-\dim(G)-2$ then $\pi_j(F_{n-1}^G(\mathbb{S}^k\backslash GQ_1))=0$.
\end{enumerate}}


The second main result of Section~3 is the following proposition that gives a more detailed description of the  long exact sequence in homotopy of the fibration associated with the inclusion map $i_{n}(\mathbb{S}^k/G): F_n(\mathbb{S}^k/G)\lhra\prod_1^n\mathbb{S}^k/G$.

\medskip

\noindent
{\bf Proposition~\ref{diag}.} {\em Let $k\geq 3$ be odd, let $j\geq 2$, and let $G\times \mathbb{S}^k\longrightarrow\mathbb{S}^k$ be a free action of a finite group $G$ on $\mathbb{S}^k$. Let $\Delta_j^n(\mathbb{S}^k/G)$ denote the diagonal subgroup of $\prod_1^n\pi_j(\mathbb{S}^k/G)$.
Then:
\begin{enumerate}[(1)]
\item the image of the homomorphism $\pi_j(i_n({\mathbb{S}^k/G})):\pi_j(F_n(\mathbb{S}^k/G))
\longrightarrow \prod_1^n\pi_j(\mathbb{S}^k/G)$ is the diagonal subgroup $\Delta_j^n(\mathbb{S}^k/G)$.

\item there are split short exact sequences of the form:
\begin{gather*}
0\longrightarrow \Delta_j^n(\mathbb{S}^k/G) \longrightarrow \pi_j\biggl(\prod_1^n\mathbb{S}^k/G\biggr)\longrightarrow \pi_{j-1}\biggl(\prod_1^{n-1}\Omega(\mathbb{S}^k/G)\biggr)\longrightarrow 0 \quad\text{and}\\
0\longrightarrow  \pi_j(F_{n-1}^G(\mathbb{S}^k\backslash GQ_1)) \longrightarrow  \pi_j(F_n(\mathbb{S}^k/G)) \longrightarrow \Delta_j^n(\mathbb{S}^k/G) \longrightarrow 0.
\end{gather*}
In particular, if $G$ is the trivial group then there is a split short exact sequence:
$$0\longrightarrow  \pi_j(F_{n-1}(\mathbb{R}^k)) \longrightarrow  \pi_j(F_n(\mathbb{S}^k)) \longrightarrow \Delta_j^n(\mathbb{S}^k) \longrightarrow 0.$$
\end{enumerate}}

\bigskip

\noindent
\textbf{Acknowledgements.} The authors would like to thank the referee for having carefully read the paper and for his or her suggestions. The first and second authors are deeply indebted to the Institute for Mathematical Sciences, National University of Singapore, for support to attend the \textit{Combinatorial and Toric Homotopy Conference} (Singapore, 23\textsuperscript{rd}--29\textsuperscript{th} August 2015) in honour of Frederick Cohen, where many of the ideas for this paper originated. Further work on the paper took place during the visit of the second author to the Institute of Mathematics, Kazimierz Wielki University, Bydgoszcz (Poland) during the period 31\textsuperscript{st}~August--12\textsuperscript{th}~September 2015, and during the visit of the third author to the Departamento de Matem\'atica do IME~--~Universidade de S\~ao Paulo (Brasil) during the period 19\textsuperscript{th}~October--3\textsuperscript{rd}~November 2015. These two visits were partially supported by  FAPESP-Funda\c c\~ao de Amparo a Pesquisa do Estado de S\~ao Paulo, Projeto Tem\'atico Topologia Alg\'ebrica, Geom\'etrica 2012/24454-8 (Brasil),  and by the CNRS/FAPESP programme n\textsuperscript{o}~226555 (France) and n\textsuperscript{o}~2014/50131-7 (Brasil) respectively.

\newpage

\setcounter{section}{1}
\setcounter{equation}{0}
\noindent
{\large\bf 1.\ Preliminaries}

\medskip

In this section, we first prove a lemma concerning the pull-back of a covering map. We then go on to prove some results about homotopy fibres of homotopy commutative diagrams that will be useful in what follows.

\begin{lemX}\label{L}
Let $(X, x_0)$ and $(Y, y_0) $ be pointed  topological spaces. Given a map $f : (X, x_0)\longrightarrow (Y, y_0)$ and a covering map $p :(\tilde{Y},\tilde y_0)\longrightarrow (Y, y_0)$ for a path-connected space $\tilde{Y}$,
 consider the associated pull-back square:
\begin{equation}\label{pullback}
\begin{gathered}
\xymatrix{X\times_Y\tilde{Y}\ar[d]_{p'}\ar[rr]^{f'}&& \tilde{Y}\ar[d]^p\\
X\ar[rr]_f&&Y,}
\end{gathered}
\end{equation}
where $p'$ and $f'$ are the projections of $X\times_Y\tilde{Y}$ onto the first and second factors, respectively. Then $p' :X\times_Y\tilde{Y}\longrightarrow X$ is a covering map, and $\pi_1(p')(\pi_1(X\times_Y\tilde{Y}))=\pi_1(f)^{-1}(\pi_1(p)\pi_1(\tilde{Y}))$. 
If further $X$ is path connected and $p : (\tilde{Y},\tilde{y}_0)\longrightarrow (Y,y_0)$ is the universal covering map then the number 
$\#\pi_0(X\times_Y\tilde{Y})$ is equal to the index $[\pi_1(Y):\pi_1(f)(\pi_1(X))]$ of the image
$\pi_1(f)(\pi_1(X))$ in the group $\pi_1(Y)$. If in addition $\pi_1(f) : \pi_1(X)\longrightarrow \pi_1(Y)$ is an epimorphism then $p' :X\times_Y\tilde{Y}\longrightarrow X$ is also the universal covering map.
\end{lemX}

\begin{proof}
The fact that the induced map $p' : X\times_Y\tilde{Y}\longrightarrow X$ is a covering is well known. 
Then the  proof of  \cite[Chapter V, Proposition 11.1]{massey} leads to $\pi_1(p')(\pi_1(X\times_Y\tilde{Y}))=
\pi_1(f)^{-1}(\pi_1(p)\pi_1(\tilde{Y}))$.


Now suppose that  $X$ is path connected and that $p : (\tilde{Y},\tilde{y}_0)\longrightarrow (Y,y_0)$ is the universal covering map. We shall show that the set $\pi_0(X\times_Y\tilde{Y})$ of path-connected components of $X\times_Y\tilde{Y}$ is in bijection with the set of right $\pi_1(f)(\pi_1(X))$-cosets in $\pi_1(Y)$.
First, notice that the restriction of the map $f' : X\times_Y\tilde{Y}\longrightarrow \tilde{Y}$ to $p'^{-1}(x_{0})$ induces a bijection $\varphi : p'^{-1}(x_{0})\longrightarrow p^{-1}(y_{0})$.
In light of the canonical actions $\pi_1(X)\times p'^{-1}(x_{0})\longrightarrow p'^{-1}(x_{0})$ and $\pi_1(Y)\times p^{-1}(y_{0})\longrightarrow p^{-1}(y_{0})$, we have:
\begin{equation*}
p'^{-1}(x_{0})=\bigsqcup_{i\in \pi_0(X\times_Y\tilde{Y})} \pi_1(X)(x_0,\tilde{y}_i)\;\mbox{and}\;p^{-1}(y_{0})=\pi_1(Y)y_0.
\end{equation*}
Since the fibre $p'^{-1}(x_{0})$ consists of the points of the form $(x_0, \tilde y)$, where $\tilde y\in p^{-1}(y_0)$, for any $(x_0, \tilde y)\in p'^{-1}(x_{0})$ there exists $\alpha_{\bar{y}}\in\pi_1(Y)$ such that
$\bar{y}=\alpha_{\bar{y}}y_0$. 
Consequently, $\varphi(\pi_1(X)(x_0,\tilde{y}_i))=((\pi_1(f)(\pi_1(X)))\alpha_{\bar{y}_i})y_0$ for $i\in \pi_0(X\times_Y\tilde{Y})$ which shows that two elements $(x_0, \tilde{y}'),(x_0, \tilde{y}'')\in p'^{-1}(x_{0})$ belong to the same path component of $X\times_Y\tilde{Y}$ if and only if $(\pi_1(f)(\pi_1(X)))\alpha_{\bar{y}'}=(\pi_1(f)(\pi_1(X)))\alpha_{\bar{y}''}$.
\par To prove the last part of the statement, by the previous parts, $X\times_Y\tilde{Y}$ is a connected covering space of $X$, and $\pi_1(p')(\pi_1(X\times_Y\tilde{Y}))$ is trivial, so $\pi_1(X\times_Y\tilde{Y})$ is trivial because $p'$ is a covering map. It follows that $X\times_Y\tilde{Y}$ is the universal covering space of $X$ as required.
\end{proof}

Let
\begin{equation*}
\text{$E_f=\{(x,\gamma)\in X\times Y^I \mid \gamma(0)=f(x)\}$ and $I_f=\{ (x,\gamma)\in E_{f} \mid \gamma(1)=y_{0} \}$}
\end{equation*}
denote the \emph{mapping path} and \emph{homotopy fibre} of $f$, respectively.
Notice that $I_f$ is determined by the homotopy pull-back diagram,
\begin{equation}\label{eq:homofibrepb}
\begin{gathered}
$$\xymatrix{I_f\ar[d]^q\ar[rr]^{p_{f}} && X\ar[d]^f\\
\ast\ar[rr]^{c_{y_0}}&&Y,}$$
\end{gathered}
\end{equation}
where $c_{y_0} :\ast\longrightarrow Y$ is the constant map determined by a point $y_0\in Y$, $q:I_{f}\longrightarrow \ast$ is the constant map, and $p_{f}:I_f\longrightarrow X$ is the projection map. It is well known that $I_f\lhra E_f \longrightarrow Y$ is a fibration and that $E_f$ and $X$ have the same homotopy type~\cite[Proposition~3.5.8 and Remark~3.5.9]{A}. We will refer to the sequence of maps $I_{f}\stackrel{p_{f}}{\longrightarrow} X\stackrel{f}{\longrightarrow} Y$ as a \emph{homotopy fibration}.
By~\cite[Proposition~3.3.17]{A},
the homotopy fibres of two homotopic maps have the same homotopy type.

Recall that a map $f : X\longrightarrow Y$ of pointed spaces is said to be \emph{$n$-connected} if the induced homomorphism
$\pi_{j}(f): \pi_{j}(X) \longrightarrow \pi_{j}(Y)$ is surjective if $j=n$ and is an isomorphism for all $j\leq n-1$.
Note that the map $f$ is $n$-connected if and only if its homotopy fibre $I_f$ is $(n-1)$-connected, in other words,
if $\pi_{i}(I_{f})=0$ for all $0\leq i\leq n-1$.

{\begin{RemX}\label{defIgg}\mbox{}
\begin{enumerate}[(a)]
\item\label{it:deflgga} Suppose that the square
\begin{equation}\label{eq:homoco}
\begin{gathered}
\xymatrix{X\ar[d]^g\ar[rr]^f && Y \ar[d]^{g'}\\
X'\ar[rr]^{f'}&&Y'}
\end{gathered}
\end{equation}
is homotopy-commutative and that $g' f(x_{0})=f' g(x_{0})$ for some $x_{0}\in X$.
Let $H : X\times I\longrightarrow Y'$ be a homotopy from
$g'\circ f$ to $f'\circ g$ such that $H(x_0,t)=g' f(x_{0})=f' g(x_{0})$ for all $t\in I$, and
consider the associated map $\varphi : X \longrightarrow (Y')^{I}$ is given by $\varphi(x)(t)=H(x,t)$ for all $x\in X$ and $t\in I$.
In particular, $\varphi(x)(0)= g'\circ f(x)$ and $\varphi(x)(1)= f'\circ g(x)$.
Then there is a homotopy-commutative diagram:
\begin{equation}\label{pullbacksquare}
\begin{gathered}
\xymatrix{X\ar@/_/[ddr]_f\ar@/^/[drrr]^g\ar[dr]^h\\
& P\ar[d]^{\bar{f}}\ar[rr]_{\bar{g}}&& X'\ar[d]^{f'}\\
&Y\ar[rr]_{g'}&&Y',}
\end{gathered}
\end{equation}
where
\begin{equation*}
P=\{ (y,\gamma,x') \in Y \times (Y')^{I} \times X' \mid \text{$\gamma(0)=g'(y)$ and $\gamma(1)=f'(x')$} \}
\end{equation*}
is the standard homotopy pull-back $P$, the maps $\bar{f}$ and $\bar{g}$ are defined by projection onto the first and third factors respectively, and
$h:X\longrightarrow P$ is the corresponding whisker map given by $h(x)= (f(x), \varphi(x), g(x))$ for all $x\in X$. If $y_0=f(x_0)$ and $y'_0=g'f(x_0)$ then:
\begin{equation}\label{eq:twosquare}
\begin{gathered}
\xymatrix{I_{f}\ar[d]^q\ar[rr]^{p_{f}} && X\ar[d]^{f}\\
\ast\ar[rr]^{c_{f(x_0)}}&&Y}\quad\raisebox{-0.7cm}{\text{and}}\quad\xymatrix{I_{f'}\ar[d]^{q'}\ar[rr]^{p_{f'}} && X'\ar[d]^{f'}\\
\ast\ar[rr]^{c_{f'g(x_0)}}&&Y'}
\end{gathered}
\end{equation}
are the standard homotopy pull-backs. The map $H' : I_f\times I\longrightarrow Y$ defined by $H'((x,\gamma),t)=\gamma(t)$ for all $(x,\gamma)\in I_f$ and $t\in I$ is a homotopy
from $fp_f$ to $c_{f(x_0)}q$,
and the associated path $\varphi' : I_f\longrightarrow Y^I$ is given by $\varphi'(x,\gamma)(t)=\gamma(t)$.
The square
\begin{equation}\label{eq:squareif}
\begin{gathered}
\xymatrix{I_f\ar[d]^q\ar[rr]^{gp_f} && X' \ar[d]^{f'}\\
\ast\ar[rr]^{c_{f'g(x_0)}}&&Y'}
\end{gathered}
\end{equation}
obtained by concatenating the square~(\ref{eq:homoco}) and the left-hand square of~(\ref{eq:twosquare})
is homotopy commutative, and a homotopy $G : I_f\times I\longrightarrow Y'$ is determined by the concatenation of the
homotopies $\bar{H}\circ (p_f\times \mbox{id}_{I}) : I_f\times I\longrightarrow Y'$ and $g' \circ H' : I_f\times I\longrightarrow Y'$, where $\bar{H}(x,t)=H(x,1-t)$ for all $x\in X$ and $t\in I$. The map $\psi : I_f\longrightarrow (Y')^I$ associated to $G$ is given by $\psi(x,\gamma)=\varphi(x)^{-1}\ast (g'\circ \gamma)$ for all $(x,\gamma)\in I_f$, and the whisker map $\mathcal{I}(g,g') : I_{f}\longrightarrow I_{f'}$ corresponding to the right-hand homotopy pull-back of~(\ref{eq:twosquare}) and the square~(\ref{eq:squareif}) and defined by $\mathcal{I}(g,g')(x,\gamma)=(g(x),\varphi(x)^{-1}\ast (g'\circ \gamma))$ for
all $(x,\gamma)\in I_f$ makes the following diagram
$$\xymatrix{I_{f} \ar[rr]^{p_{f}} \ar@{.>}[d]^{\mathcal{I}(g,g')} && X\ar[d]^g\ar[rr]^f && Y \ar[d]^{g'}\\
I_{f'} \ar[rr]^{p_{f'}} && X'\ar[rr]^{f'}&&Y'}$$
homotopy-commutative. Note that the construction of $\mathcal{I}(g,g')$ corresponds to that given in~\cite[pp.~225-226]{ma}.

\item\label{it:deflggabis} With the notation of part~(\ref{it:deflgga}), if the maps $g : X\longrightarrow X'$ and $g' : Y\longrightarrow Y'$ are homotopy equivalences then by applying~\cite[Lemma 45]{ma} to the following homotopy-commutative diagram:
$$\xymatrix{X\ar[d]^g\ar[rr]^f && Y \ar[d]^{g'}&&\ar[ll]\ast\ar@{=}[d]\\
X'\ar[rr]^{f'}&&Y'&&\ar[ll]\ast,}$$
we see that the map $\mathcal{I}(g,g') : I_f\longrightarrow I_{f'}$ is also a homotopy equivalence. In particular,
if $g : X\longrightarrow X'$ (resp.\ $g' : Y\longrightarrow Y'$) is a homotopy equivalence,
it follows from the commutative diagrams
\begin{equation*}
\xymatrix{X\ar@{=}[d]\ar[rr]^f && Y \ar[d]^{g'}\\
X\ar[rr]^{g'f}&&Y'}\quad\raisebox{-0.7cm}{\text{and}}\quad\xymatrix{X\ar[d]^g\ar[rr]^{f'g} && Y' \ar@{=}[d]\\
X'\ar[rr]^{f'}&&Y'}
\end{equation*}
that the map $\mathcal{I}(\mbox{id}_X,g') : I_f\longrightarrow I_{g'f}$ (resp.\ $\mathcal{I}(g,\mbox{id}_Y') : I_{f'g}\longrightarrow I_{f'}$) is a homotopy equivalence.
\item\label{it:deflggb} If diagram~(\ref{eq:homoco}) is commutative and not just homotopy commutative then we will always take $H$ to be constant
\emph{i.e.} $H(x,t)=f'\circ g(x)$ for all $(x,t)\in X\times I$, and in this case, $\mathcal{I}(g,g')(x,\gamma)=(g(x), g'\circ \gamma)$, from which
it follows that $g\circ p_{f}=p_{f'}\circ \mathcal{I}(g,g')$. Further, $\mathcal{I}(g,g')$ is the restriction to the homotopy fibres $I_{f}$ and
$I_{f'}$ of the map $X \times Y^{I} \longrightarrow X' \times (Y')^{I}$ that sends $(x,\gamma)\in X \times Y^{I}$ to
$(g(x), g'\circ \gamma)\in X' \times (Y')^{I}$.

\item\label{it:deflggc} By~\cite[Lemma 6.4.11]{A}, a map $ f : X\longrightarrow Y$ is $n$-connected if and only if $\operatorname{\text{conn}}(I_f)\ge n-1$.

\item\label{it:deflggd} Let $f:X \longrightarrow Y$ be a pointed map, and consider the associated homotopy fibration $I_{f} \stackrel{p_{f}}{\longrightarrow} X
\stackrel{f}{\longrightarrow} Y$. The boundary map $d: \Omega(Y)  \longrightarrow I_{f}$ is defined by $d(\omega)=(x_{0}, \omega)$
for all $\omega \in \Omega(Y)$, where $\Omega(Y)$ denotes the loop space of Y. Suppose further that $f$ is null homotopic. Setting $g=\text{id}_{X}$, $g'=\text{id}_{Y}$ and $f'=c_{y_{0}}$ in
the construction of part~(\ref{it:deflgga}), and taking the homotopy $H$ to be such that $H(x_{0},t)=y_{0}$ for all $t\in I$, we have
$\varphi(x_{0})=c_{y_{0}}$, and the map $\mathcal{I}(\text{id}_{X}, \text{id}_{Y}): I_{f} \longrightarrow X \times \Omega(Y)$
is a homotopy equivalence defined by $\mathcal{I}(\text{id}_{X}, \text{id}_{Y})(x,\gamma)= (x, \varphi(x)^{-1}\ast \gamma)$ for all
$(x,\gamma)\in I_{f}$. Thus $\mathcal{I}(\text{id}_{X}, \text{id}_{Y})\circ d: \Omega(Y) \longrightarrow X \times \Omega(Y)$
is given by $\mathcal{I}(\text{id}_{X}, \text{id}_{Y})\circ d(\omega)=(x_{0}, \omega)$ for $\omega\in\Omega (Y)$, and so coincides with  inclusion of $\Omega (Y)$.
This yields a homotopy equivalence:
\begin{equation}\label{E1}I_f\stackrel{\simeq}{\longrightarrow}X\times\Omega(Y).
\end{equation}
If $F\lhra X\longrightarrow Y$ is a fibration then it is well known that the homotopy fibre of
the inclusion map $F\lhra X$ has the homotopy type of $\Omega(Y)$. Consequently, if the inclusion map $F\lhra X$ is null homotopic then by~(\ref{E1})
there is a homotopy equivalence:
\begin{equation}\label{E2} \Omega (Y)\stackrel{\simeq}{\longrightarrow}F\times\Omega(X).
\end{equation}
\end{enumerate}
\end{RemX}}

\begin{lemX}[{\cite[Appendix]{GG11}}]\label{crabb}
Let $X\stackrel{f}{\longrightarrow}Y\stackrel{g}{\longrightarrow}Z$ be pointed maps,
let $\mathcal{I}(f,\mbox{\em\small id}_Z) : I_{gf}\longrightarrow I_g$ be the map defined in Remark~\ref{defIgg}(\ref{it:deflggb}),
and let $p_{gf}: I_{gf} \longrightarrow X$ and $p_{g}: I_{g} \longrightarrow Y$ denote the respective projections. Then the induced map
$\mathcal{I}(p_{gf},p_{g}): I_{\mathcal{I}(f,\mbox{\em\tiny id}_Z)} \longrightarrow I_f$ is a homotopy equivalence.\label{C}
\end{lemX}
The statement of Lemma~\ref{crabb} may be summed up by the three rightmost columns of the following commutative diagram:

\begin{equation}\label{eq:crabb}
\begin{gathered}
\xymatrix{%
I(f') \ar@{.>}[d] \ar[rr]^{\mathcal{I}(\iota_{gf},\iota_{g})}_-{\simeq} && I_{\mathcal{I}(f,\text{\tiny id}_Z)} \ar[rr]^{\mathcal{I}(p_{gf},p_{g})}_-{\simeq} \ar@{.>}[d] && I_{f} \ar@{.>}[d] &\\
F_{gf} \ar@{^{(}->}[rr]^{\iota_{gf}}_{\simeq} \ar[d]^{f'} && I_{gf} \ar@{.>}[rr]^{p_{gf}} \ar[d]^{\mathcal{I}(f,\text{\tiny id}_Z)} && X \ar[rr]^{gf} \ar[d]^{f} && Z \ar@{=}[d]\\
F_{g} \ar@{^{(}->}[rr]^{\iota_{g}}_{\simeq} && I_{g} \ar@{.>}[rr]^{p_{g}} && Y \ar[rr]^{g} && Z,}
\end{gathered}
\end{equation}
where the dotted arrows indicate the corresponding homotopy fibres.

Let $f : X\longrightarrow Y$ be a fibration, let $F_f=f^{-1}(y_0)$ be the topological fibre over a point $y_{0} \in Y$, and let $I_{f}$ be the homotopy
fibre of $f$. Then, by \cite[Proposition~3.5.10]{A}, the injective map $\iota_{f}: F_f\lhra I_f$ defined by $\iota_{f}(x)=(x,c_{y_{0}})$ is a homotopy
equivalence. Let $j_{f}: F_{f} \lhra X$ denote the inclusion of the topological fibre in the total space $X$. Note that $j_{f}=p_{f}\circ \iota_{f}$.
We obtain the following corollary (summarised in diagram~(\ref{eq:crabb})) by combining this with Remark~\ref{defIgg}(\ref{it:deflggabis}) and Lemma~\ref{crabb}.

\begin{CorX}\label{crabb1}
Let $X\stackrel{f}{\longrightarrow}Y\stackrel{g}{\longrightarrow}Z$ be pointed maps, and suppose that $X\stackrel{gf}{\longrightarrow}Z$ and $Y\stackrel{g}{\longrightarrow}Z$ are fibrations with respective topological fibres $F_{gf}$ and $F_g$. If $f': F_{gf} \longrightarrow F_g$ is the restriction of the map $f : X\longrightarrow Y$ to $F_{gf}$ and $F_g$ then the map $\mathcal{I}(j_{gf},j_{g}): I_{f'}\longrightarrow I_f$ is a homotopy equivalence.
\end{CorX}

As a consequence of Lemma~\ref{C}, we have the following result (\emph{cf.}\ \cite[Chapter IX]{Hil}).
\begin{proX} \label{LL}
Given the homotopy-commutative square (\ref{eq:homoco}),
consider the associated homotopy-commutative diagram~(\ref{pullbacksquare}).
Then there exist fibrations $\hat I_h\longrightarrow \hat I_f\longrightarrow \hat I_{f'}$ and $\hat I_h\longrightarrow \hat I_g\longrightarrow \hat I_{g'}$, where for each $k\in \{ f,g,f',g',h\}$, $\hat I_{k}$ is a space that has the same homotopy type as the homotopy fibre $I_{k}$ of the map $k$. Further, the homotopy fibres $I_{\mathcal{I}(g,g')}$, $I_{\mathcal{I}(f,f')}$ and $I_h$ have the same homotopy type.
\end{proX}

\begin{proof}
We use the notation of Remark~\ref{defIgg}(\ref{it:deflgga}). Since $f= \bar{f} h$,
applying Lemma~\ref{C} to the composition $X\stackrel{h}{\longrightarrow} P\stackrel{\bar{f}}
{\longrightarrow} Y$, we obtain a fibration $\hat{I}_{h} \longrightarrow \hat{I}_{f} \longrightarrow I_{\bar f}$,
where $\hat{I}_{h}=I_{\mathcal{I}(h,\text{\tiny id}_{Y})}\simeq I_{h}$ and $\hat I_f=E_{\mathcal{I}(h,\text{\tiny id}_{Y})} \simeq I_{\bar{f}h} =I_{f}$, which proves the first part of the statement for the first fibration.

We now prove that the homotopy fibres $I_{\mathcal{I}(g,g')}$ and $I_h$ have the same homotopy type.
Given a point $y_0\in Y$, the two squares of the following diagram:
$$\xymatrix{I_{\bar{f}}\ar[d]\ar[rr]^{p_{\bar{f}}}&& P\ar[d]^{\bar{f}}\ar[rr]^{\bar{g}}&&X'\ar[d]^{f'}\\
\ast\ar[rr]^{c_{y_0}}&&Y\ar[rr]^{g'}&&Y'}$$
are homotopy pull-backs, and thus the whole rectangle is a homotopy pull-back~\cite[Lemma 12]{ma}.
But the square
$$\xymatrix{I_{f'}\ar[d]\ar[rr]^{p_{f'}} && X'\ar[d]^{f'}\\
\ast\ar[rr]^{c_{g'(y_0)}}&&Y'}$$ is the standard homotopy pull-back, and $\mathcal{I}(\bar g,g') : I_{\bar{f}}\longrightarrow I_{f'}$ is a whisker map by Remark~\ref{defIgg}(\ref{it:deflgga}). It follows from~\cite[p.~226]{ma} that $\mathcal{I}(\bar g,g')$ is a homotopy equivalence. Now $\mathcal{I}(g,g')= \mathcal{I}(\bar{g}h,g')=\mathcal{I}(\bar{g},g')\circ \mathcal{I}(h,\mbox{id}_Y)$, and using the fact that
$\mathcal{I}(\bar{g},g')$ is a homotopy equivalence, it follows from the last part of Remark~\ref{defIgg}(\ref{it:deflggabis}) that $\mathcal{I}(\text{\tiny id}_{I_{f}}, \mathcal{I}(\bar{g},g')): I_{\mathcal{I}(h,\text{\tiny id}_Y)}\longrightarrow  I_{\mathcal{I}(g,g')}$ is a homotopy equivalence. Further, $I_{\mathcal{I}(h,\text{\tiny id}_Y)}\simeq I_{h}$ by Lemma~\ref{crabb}, from which we conclude that $I_{\mathcal{I}(g,g')}\simeq I_{h}$.
The remaining assertions of the proposition follow by exchanging the r\^{o}les of $f,\bar{f}$ and $f'$ with those of $g,\bar{g}$ and $g'$ respectively.
\end{proof}

To prove the final result of this section, we require the following lemma. 

\begin{lemX}\label{hpull}
Let $f : X\longrightarrow Y$ be a pointed map.

\begin{enumerate}[(1)]
\item\label{it:hpull1} If $f$ is a homotopy equivalence then the homotopy fibre $I_f$ of $f$ is contractible.
\item\label{it:hpull2} If $I_f$ is contractible then the induced map $\Omega f : \Omega X\longrightarrow \Omega Y$ is a homotopy equivalence.
In particular, if the induced map $\pi_0(f) : \pi_0(X)\longrightarrow \pi_0(Y)$ is surjective then $f : X\longrightarrow Y$ is a weak homotopy equivalence.
\end{enumerate}
\end{lemX}

The first part of the lemma follows by applying~\cite[Lemma~45]{ma} to the homotopy pull-back diagram~(\ref{eq:homofibrepb}). The second part may be obtained by noting that if $Z$ is a pointed space then the Puppe fibre sequence gives rise to a bijection $(\Omega f)_{\ast}: [Z,\Omega X] \longrightarrow [Z,\Omega Y]$, which implies that $\Omega f: \Omega X \longrightarrow \Omega Y$ is a homotopy equivalence.

\begin{RemX}
Let $ f : X\longrightarrow Y$ be a pointed map of path-connected spaces such that $\Omega f : \Omega X\longrightarrow \Omega Y$ is a homotopy equivalence. Then $f$ is a weak homotopy equivalence, but it is not necessarily a homotopy equivalence. Indeed, the loop space of the Warsaw circle is contractible, but the Warsaw circle is not. Nevertheless, under certain hypotheses, $f$ is a homotopy equivalence. For example, if $f : X\longrightarrow Y$ is a pointed map of path-connected Dold spaces (the class of such spaces contains $CW$-spaces) for which $\Omega f : \Omega X\longrightarrow \Omega Y$ is a homotopy equivalence, Allaud~\cite{Al} proved that $f$ is a homotopy equivalence.
\end{RemX}

Combining diagram~(\ref{eq:crabb}) with Lemma~\ref{hpull}, we obtain the following result.

\begin{CorX} \label{LL1} Given a fibration $f: X\longrightarrow Y$ and a map $g:Y'\longrightarrow Y$,
consider the following strict pull-back diagram
$$\xymatrix{Y'\times_YX\ar[d]_{f'}\ar[rr]^-{g'}&& X\ar[d]^f\\
Y'\ar[rr]_g&&Y.}$$
Then the induced map $\Omega \mathcal{I}(f',f) : \Omega I_{g'}\longrightarrow \Omega I_g$ is a homotopy equivalence. Further, if the map  $g:Y'\longrightarrow Y$ is $1$-connected then the map $\mathcal{I}(f',f) : I_{g'}\longrightarrow I_g$
is a weak homotopy equivalence.
\end{CorX}
\begin{proof}
Let $y'_{0}\in Y'$, and set $y_{0}=g(y_{0'})\in Y$. Since $f: X\longrightarrow Y$ is a fibration, the fact that the above diagram is
a strict pull-back implies that $f' : Y'\times_YX \longrightarrow Y'$ is also a fibration. Further, the restriction of $g'$ to the topological
fibres $F_{f'}=f'^{-1}(y'_0)$ and $F_f=f^{-1}(y_0)$ induces a homeomorphism $F_{f'}\stackrel{\approx}{\longrightarrow} F_f$. Since the injective
maps $\iota_{f}: F_f\lhra I_f$ and $\iota_{f'}: F_{f'}\lhra I_{f'}$ (as defined just before the statement of Corollary~\ref{crabb1})
are homotopy equivalences, the map $\mathcal{I}(g',g) : I_{f'}\longrightarrow I_f$ is a homotopy equivalence, and so
Lemma~\ref{hpull}(\ref{it:hpull1}) implies that the homotopy fibre $I_{\mathcal{I}(g',g)}$ is contractible. So by Proposition~\ref{LL}, the homotopy fibre $I_{\mathcal{I}(f',f)}$ of the map $\mathcal{I}(f',f) : I_{g'}\longrightarrow I_g$  is also contractible, and we conclude from Lemma~\ref{hpull}(\ref{it:hpull2}) that the induced map $\Omega \mathcal{I}(f',f) : \Omega I_{g'}\longrightarrow \Omega I_g$ is a homotopy equivalence, which proves the first part of the statement. If further $g:Y'\longrightarrow Y$ is $1$-connected then $g' : Y'\times_YX\longrightarrow X$ is also $1$-connected, and so by Remark~\ref{defIgg}(\ref{it:deflggc}), the spaces $I_g$ and  $I_{g'}$ are path connected. Lemma~\ref{hpull}(\ref{it:hpull2}) implies that the map $\mathcal{I}(f',f) : I_{g'}\longrightarrow I_g$ is a weak homotopy equivalence as required.
\end{proof}



\setcounter{section}{2}
\setcounter{equation}{0}
\setcounter{thmX}{0}

\noindent
{\large\bf 2.\ Configuration spaces and the inclusion map $F_n(X)\lhra\prod_1^nX$}

\medskip

Let $X$ be a topological space, and let $n\geq 1$. Given a free action $G\times X\longrightarrow X$ of a group $G$ on $X$, recall from~(\ref{eq:defconfig}) (resp.\ from~(\ref{eq:deforbconfig})) that $F_n(X)$ is the $n$\textsuperscript{th} configuration space (resp.\ $F_n^G(X)$ is the orbit configuration space) of $X$. Let $i_n(X) : F_n(X)\lhra\prod_1^nX$ denote the inclusion map. We assume from now on that $G$ is a topological group. Let $q(X):X\longrightarrow X/G$ denote the associated quotient map, and let $\bar{x}=q(X)(x)$ for all $x\in X$. Then $q(X)$ induces a map $\psi_n(X): F_n^G(X)\longrightarrow F_n(X/G)$ given by $\psi_n(X)(x_1,\ldots,x_n)=(\bar{x}_1,\ldots,\bar{x}_n)$. This map is well defined, since if $(x_{1},\ldots,x_{n})\in F^G_n(X)$ then
for all $i\neq j$, $Gx_{i}\cap Gx_{j}= \emptyset$, so $\bar{x}_{i}\neq \bar{x}_{j}$.
We thus obtain the following commutative diagram:
\begin{equation}\label{pullbackFGnX}
\begin{gathered}
\xymatrix{F^G_n(X)\ar[d]_{\psi_n(X)} \ar@{^{(}->}[rr]^{i_{n}(X)\left\lvert_{F^G_n(X)} \right.} && \prod_1^n X\ar[d]^{\prod_1^n q(X)}\\
F_n(X/G)\ar@{^{(}->}[rr]_{i_n(X/G)}&&\prod_1^n X/G.}
\end{gathered}
\end{equation}
Further, the associated (strict) pull-back $F_n(X/G)\times_{\prod_1^nX/G}\prod_1^nX$ is given by:
\begin{equation*}
\biggl\{  \biggl. \bigl( (\bar{x}_{1}, \ldots, \bar{x}_{n}), (y_{1},\ldots,y_{n}) \bigr)\in F_n(X/G) \times \prod_1^nX \,\biggr\rvert\, \text{$\bar{x}_{i}= \bar{y}_{i}$ for all $1\le i\le n$} \biggr\}.
\end{equation*}
One may then check that the map $F_n(X/G)\times_{\prod_1^nX/G}\prod_1^nX \longrightarrow F_{n}^{G}(X)$ given by
$\bigl( (\bar{x}_{1}, \ldots, \bar{x}_{n}), (x_{1},\ldots,x_{n}) \bigr)\longmapsto (x_{1},\ldots, x_{n})$ is a (well-defined) homeomorphism. In particular, (\ref{pullbackFGnX}) is a (strict) pull-back diagram.

Now suppose that $X$ is a topological manifold, let $Q_r =\{q_1,\ldots,q_r\}$ be a finite, non-empty subset of $X$ whose elements belong
to distinct orbits, and let $\bar{Q}_{r}=q(X)(Q_{r})$. Let  $i_{n}'(X) : F_{n}(X\backslash Q_{r}) \longrightarrow \prod_{1}^{n} X$
(resp.\ $i_{n}''(X) : F_n^G(X\backslash GQ_r) \longrightarrow \prod_{1}^{n} X$) denote the composition of $i_{n}(X\backslash Q_{r})$
(resp.\ of $i_{n}(X\backslash Q_{r})\left\lvert_{F_n^G(X\backslash GQ_r)}\right.$) with the inclusion map
$\prod_{1}^{n} X\backslash Q_{r}\lhra \prod_{1}^{n} X$. Since the diagram
$$\begin{gathered}
\xymatrix{\prod_1^n (X\backslash GQ_r) \ar@{^{(}->}[rr]\ar[d]_{q(X\backslash GQ_r)} && \prod_1^n X\ar[d]^{\prod_1^n q(X)}\\
\prod_1^n ((X/G)\backslash \bar Q_r)\ar@{^{(}->}[rr]&&\prod_1^n X/G}
\end{gathered}$$
is a (strict) pull-back, taking $X \backslash GQ_{r}$ in place of $X$ in diagram~(\ref{pullbackFGnX}) and composing with the above diagram, we obtain the following (strict) pull-back:
\begin{equation}\label{E}
\begin{gathered}
\xymatrix{F^G_n(X\backslash GQ_r)\ar[d]_{\psi_n(X\backslash GQ_r)}\ar@{^{(}->}[rrr]^{i''_{n}(X)}&&& \prod_1^nX\ar[d]^{\prod_1^n q(X)}\\
F_n((X/G)\backslash\bar{Q}_r)\ar@{^{(}->}[rrr]_{i'_n(X/G)}&&&\prod_1^nX/G.}
\end{gathered}
\end{equation}
Now the map $\prod_1^nq(X) : \prod_1^nX\longrightarrow \prod_1^nX/G$ is a fibration, and the application of Corollary~\ref{LL1} to the map $\mathcal{I}\Bigl( \psi_n(X\backslash GQ_r), \prod_1^n q(X)\Bigr): I_{i''_n(X)}\longrightarrow I_{i'_{n}(X/G)}$ gives rise to the following homotopy equivalence:
\begin{equation}
\Omega\mathcal{I}\Bigl(\psi_n(X\backslash GQ_r), \prod_1^n q(X)\Bigr):\Omega I_{i''_n(X)}\stackrel{\simeq}{\longrightarrow}
\Omega I_{i'_{n}(X/G)}.\label{E-11}
\end{equation}

If $1\le i\le n$, let $p_i(X) : \prod_1^n X\longrightarrow \prod_1^{n-1}X$ denote the projection map given by forgetting the $i$\textsuperscript{th} coordinate. Let $J = (i_1,\ldots,i_r)$ be a sequence of integers that satisfy $1\le i_1 < i_2 < \cdots < i_r\le n$, and let $p_J(X): \prod_1^n X \longrightarrow \prod_1^{n-r} X$ be the map defined by $p_J(X)=p_{i_1}(X)\circ\cdots\circ p_{i_r}(X)$.
The map $p_{i}(X)$ restricts to the map $p_i(X)\left\lvert_{F_{n}(X)}\right. : F_n(X)\longrightarrow F_{n-1}(X)$ that is a fibration whose fibre
may be identified with $X\backslash Q _{n-1}$, where $Q_{n-1}=\{ q_{1},\ldots,q_{i-1},q_{i+1},\ldots, q_{n} \}$. This is a special case of the following result.
\begin{thmX}[{\cite[Theorem~3]{FaN}}]\label{FC}
Let $X$ be a topological manifold without boundary. With the above notation, the map $p_J(X)\left\lvert_{F_{n}(X)}\right. : F_n(X)\longrightarrow F_{n-r}(X)$ is a fibration
whose fibre may be identified with $F_r(X\backslash Q_{n-r})$, where $Q_{n-r}=Q_{n} \backslash \{ q_{i_{1}},\ldots, q_{i_{r}}\}$.
\end{thmX}

For $1\leq i\leq n$, let $J_i=(1,\ldots,i-1,i+1,\ldots,n)$. From now on, we assume that $X$ is a pointed topological manifold without boundary.
We thus obtain the following commutative diagram:
\begin{equation}\label{E'}
\begin{gathered}
\xymatrix{F_{n-1}(X\backslash Q_1)\ar@{^{(}->}[rr]^-{j_{n}'(X)}\ar@{^{(}->}[d]^{i'_{n-1}(X)}&& F_n(X)\ar[rr]^-{p_{J_{i}(X)}\left\lvert_{F_{n}(X)}\right.}\ar@{^{(}->}[d]^{i_n(X)} && X\ar@{=}[d]\\
\prod_1^{n-1}X\ar@{^{(}->}[rr]^-{j_{n}(X)} && \prod_1^nX\ar[rr]^-{p_{J_i}(X)} &&X,}
\end{gathered}
\end{equation}
where the two inclusions $j_{n}'(X): F_{n-1}(X\backslash Q_1) \lhra F_n(X)$ and $j_{n}(X): \prod_1^{n-1}X \lhra \prod_1^{n}X$
are given by inserting the point $q_{i}$ in the $i$\textsuperscript{th} position. Now
the rows of~(\ref{E'}) are fibrations by Theorem~\ref{FC}, so Corollary~\ref{crabb1} implies that the map:
\begin{equation}\label{E'-11}
\mathcal{I}(j_{n'}(X),j_{n}(X)) : I_{i'_{n-1}(X)}\stackrel{\simeq}{\longrightarrow}I_{i_n(X)}
\end{equation}
is a homotopy equivalence for all choices of $1\le i\le n$ above.
\par Our next goal is to prove the following theorem that is a stronger version of~\cite[Theorem 1]{Bi1},
without the hypothesis that $X$ is smooth.
\begin{thmX}\label{B}
Let $X$ be a connected, topological manifold without boundary. Then the inclusion map $i_n(X) : F_n(X)\lhra\prod_1^nX$
is $(\dim(X)-1)$-connected.
\end{thmX}

Clearly, Theorem~\ref{B} implies~\cite[Theorem 1]{Bi1}. In order to prove Theorem~\ref{B}, we shall make use of the following lemma.

\begin{lemX}\label{bir}
Let $X$ be a connected, topological manifold without boundary, and let $Q$ be a finite non-empty subset of $X$. Then the inclusion map $i(X) : X\backslash Q\lhra X$ is $(\dim (X)-1)$-connected.
\end{lemX}

\begin{proof}
If $\dim(X)=1$, the result is clear since $X$ is connected. So assume that $\dim(X)\geq 2$, and suppose first that $Q$ consists of a single point $q$. Write $X$ as the union of $X \backslash Q$ and $D_q$, where $D_q$ is a small open disc whose centre is $q$ (such a disc exists because $X$ is a topological manifold), and let $x$ be a basepoint lying in
$(X \backslash Q) \cap D_q$. Now $(X \backslash Q)\cap D_q$ has the homotopy type of a sphere whose dimension is equal to $\dim(X)-1$, which is strictly positive, so by Van Kampen's Theorem, $\pi_1(X)$ is isomorphic to the quotient of $\pi_1(X \backslash Q)$ by the normal closure in $\pi_1(X \backslash Q)$ of the image of $\pi_1((X \backslash Q)\cap D_q)$ by the homomorphism induced by the inclusion $(X \backslash Q)\cap D_q \lhra X \backslash Q$, which implies in particular that the homomorphism $\pi_{1}(i(X)):\pi_1(X \backslash Q)\longrightarrow \pi_1(X)$ is surjective.
\par If $\dim(X)=2$, the result follows using the  connectedness of $X$. Now assume that $\dim(X)>2$. Then $\pi_1((X \backslash Q)\cap D_q)$ is trivial, and we conclude from the above arguments that the homomorphism $\pi_{1}(i(X)):\pi_1(X \backslash Q)\longrightarrow \pi_1(X)$ is an isomorphism. To compare the higher homotopy groups, consider the universal
covering $p:\tilde{X} \longrightarrow X$ of $X$, and set $X_1=p^{-1}(X \backslash Q)$. Identifying $X_{1}$ with the pull-back $(X\backslash Q) \times_{X} \tilde{X}$, it follows from Lemma~\ref{L} that the map $p\left\lvert_{X_1}\right. : X_1 \longrightarrow X \backslash Q$ is the
universal covering of $X \backslash Q$, in analogy with the method of~\cite[Section~2]{FH0}. By  the relative form of the Mayer-Vietoris sequence, the relative homology groups $H_i(\tilde X, X_1)$ are zero for all
$0\leq i \leq \dim(X)-1$. The Hurewicz Theorem then implies that the homomorphism $\pi_j(X_1) \longrightarrow \pi_j(\tilde X)$ is an isomorphism if $j+1<\dim (\tilde X)$ and is surjective if $j=\dim(\tilde X)-1$. So if $Q$ consists of a single point, using the fact that $\pi_{j}(X_{1}) \cong \pi_{j}(X \backslash Q)$ if $j\geq 2$, the homomorphism $\pi_j(i(X)) :\pi_j(X\backslash Q)\longrightarrow \pi_j(X)$ is an isomorphism if $j+1<\dim (\tilde X)$ and is surjective if $j=\dim(\tilde X)-1$, which proves the result in this case. In the general case, if $Q=\{ q_{1}, \ldots, q_{m} \}$, where $m\geq 2$, the result
follows by induction on $m$ and by writing the manifold $X\backslash Q$ as $(X \backslash \{ q_{1}, \ldots, q_{m-1} \})\backslash \{ q_{m} \}$.
\end{proof}

\begin{proof}[Proof of Theorem~\ref{B}] The proof is by induction on $n$. If $n=1$ then the statement clearly holds. So assume that the result is true
for some $n\geq 1$ and for any connected topological manifold without boundary. By taking the long exact sequence in homotopy of the fibrations of the
commutative diagram~(\ref{E'}), where we replace $n$ by $n+1$,
we obtain the following commutative diagram of exact sequences:
\begin{equation}\label{eq:in+1X}
\begin{gathered}
\xymatrix@C-=0.65cm@1{\pi_{j+1}(X) \ar[r]\ar@{=}[d]&\pi_j(F_{n}(X\backslash {Q}_1))\ar[r]\ar[d]^{\pi_{j}(i_{n}'(X))} &\pi_j(F_{n+1}(X))\ar[r]\ar[d]^{\pi_{j}(i_{n+1}(X))} &\pi_j(X)\ar@{=}[d] \ar[r] & \pi_{j-1}(F_{n}(X\backslash {Q}_1))\ar[d]^{\pi_{j-1}(i_{n}'(X))}\\
\pi_{j+1}(X)\ar[r] &\pi_j(\prod_1^n X)\ar[r]&\pi_j(\prod_1^{n+1}X)\ar[r]&\pi_j(X) \ar[r] & \pi_{j-1}(\prod_1^n X)}
\end{gathered}
\end{equation}
for all $j\geq 1$. If $k\geq 0$, the homomorphism $\pi_{k}(i_{n}'(X))$ may be written as the composition $\pi_{k}(F_{n}(X\backslash Q_1))\xrightarrow{\pi_{k}(i_{n}(X\backslash Q_{1}))} \pi_{k}(\prod_1^n X\backslash Q_1) \xrightarrow{\pi_{k}(\prod_{1}^{n} i(X))} \pi_{k}(\prod_1^n X)$ of the homomorphisms induced by the composition $F_{n}(X\backslash Q_1)  \xrightarrow{i_{n}(X\backslash Q_{1})}\prod_1^n X\backslash Q_1 \xrightarrow{\prod_{1}^{n} i(X)} \prod_1^n X$. If further $k\leq j$, the first (resp.\ second) homomorphism $\pi_{k}(i_{n}(X\backslash Q_{1})): \pi_{k}(F_{n}(X\backslash Q_1))\longrightarrow \pi_{k}(\prod_1^n X\backslash Q_1)$ (resp.\ $\pi_{k}(\prod_{1}^{n} i(X)):\prod_1^n \pi_{k}(X\backslash Q_1) \longrightarrow \prod_1^n \pi_{k}(X)$) is surjective if $j<\dim(X)$, and is injective if $j+1<\dim(X)$ by the induction hypothesis applied to the manifold $X \backslash Q_1$ (resp.\ by Lemma~\ref{bir}). Hence $\pi_{j}(i_{n}'(X))$ is surjective if $j<\dim(X)$, and is an isomorphism if $j+1<\dim(X)$. Applying the 5-Lemma to~(\ref{eq:in+1X}), it follows that $\pi_{j}(i_{n+1}(X))$ is surjective if $j<\dim(X)$ (resp.\ is an isomorphism if $j+1<\dim(X)$) as required.
\end{proof}

Theorem~\ref{B} gives rise to the following corollary.
\begin{CorX}\label{cc}
Let $X$ be a connected topological manifold without boundary. Then the inclusion map $i'_r(X) : F_r(X\backslash Q_{n-r})\lhra\prod_1^rX$ is $(\dim(X)-1)$-connected.
\end{CorX}

\begin{proof}
By definition, $i'_r(X)$ is the composition of $i_r(X\backslash Q_{n-r}): F_r(X\backslash Q_{n-r})\lhra \prod_1^rX\backslash Q_{n-r}$ and the inclusion $ \prod_1^rX\backslash Q_{n-r}\lhra \prod_1^rX$. It then suffices to apply Theorem~\ref{B} (resp.\ Lemma \ref{bir}) to the first (resp.\ second) map.
\end{proof}

Following~\cite{edmonds}, we say that the action of a topological group $G$ on $X$ is \emph{tame} if the orbit space $X/G$ is a topological manifold.
For instance, \cite[Theorem 7.10]{L} implies that any free, proper and smooth action of a compact Lie group $G$ on a smooth manifold $X$ without
boundary is tame.
\par Now let $G\times X\longrightarrow X$ be a free, tame action of a group $G$ on a topological manifold $X$ without boundary, and consider
the commutative diagram~(\ref{E}) and the leftmost commutative square of~(\ref{E'}), where in the second case, we replace $X$ by $X/G$ and $Q_{1}$
by $\bar Q_{1}$. Applying the construction of Remark~\ref{defIgg}(\ref{it:deflggb}) to each of these squares gives rise to the following
commutative diagram, where the rows are homotopy fibrations:
\begin{equation}\label{threefibres}
\begin{gathered}
\xymatrix{I_{i_{n-1}''(X)}\ar[rrr]^-{p_{i''_{n-1}(X)}}\ar[d]^{\mathcal{I}(\psi_{n-1}(X \backslash GQ_1), \prod_1^{n-1}q(X))} &&& F_{n-1}^G(X\backslash GQ_1)\ar[d]^{\psi_{n-1}(X \backslash GQ_1)}\ar@{^{(}->}[rr]^-{i''_{n-1}(X)}&& \prod_1^{n-1}X\ar[d]^{\prod_1^{n-1}q(X)}\\
I_{i'_{n-1}(X/G)}\ar[rrr]^-{p_{i'_{n-1}(X/G)}}\ar[d]^{\mathcal{I}(j_{n}'(X/G), j_{n}(X/G))} &&&F_{n-1}((X/G)\backslash\bar{Q}_1)\ar@{^{(}->}[d]^{j_{n}'(X/G)}\ar@{^{(}->}[rr]^(0.55){i'_{n-1}(X/G)}&& \prod_1^{n-1}X/G\ar@{^{(}->}[d]^{j_{n}(X/G)}\\
I_{i_n(X/G)}\ar[rrr]^-{p_{i_n(X/G)}}&&&F_n(X/G)\ar@{^{(}->}[rr]^{i_n(X/G)}&&\prod_1^nX/G,}
\end{gathered}
\end{equation}
the map $\mathcal{I}(\psi_{n-1}(X \backslash GQ_1), \prod_1^{n-1}q(X)): I_{i_{n-1}''(X)}\longrightarrow I_{i'_{n-1}(X/G)}$ gives rise to a homotopy equivalence $\Omega\mathcal{I}(\psi_{n-1}(X \backslash GQ_1), \prod_1^{n-1}q(X)):\Omega I_{i_{n-1}''(X)} \longrightarrow\Omega I_{i'_{n-1}(X/G)}$ by~(\ref{E-11}), and the map $\mathcal{I}(j_{n}'(X/G), j_{n}(X/G)): I_{i'_{n-1}(X/G)}\longrightarrow I_{i_n(X/G)}$ is a homotopy equivalence by~(\ref{E'-11}).
Let $k_{n}(X): I_{i_{n-1}''(X)} \longrightarrow I_{i_n(X/G)}$ be the weak homotopy equivalence defined by $k_{n}(X)= \mathcal{I}(j_{n}'(X/G), j_{n}(X/G))\circ \mathcal{I}(\psi_{n-1}(X \backslash GQ_1),\Pi_1^{n-1}q(X))$.

\medskip

Let $G\times X\longrightarrow X$ be a tame, free action of a Lie group $G$ on a topological manifold $X$ without boundary. This action gives
rise to a fibration $G\lhra X\longrightarrow X/G$ of manifolds, where the fibre over the point $\bar Q_1\in X/G$ is identified with
$GQ_{1}$. If the inclusion map $GQ_1\lhra X$ is null homotopic then the homomorphism $\pi_j(X) \longrightarrow \pi_j(X/G)$ induced by the quotient
map $X\longrightarrow X/G$ is injective for all $j\ge 0$, so we may regard $\pi_j(X)$ as a subgroup of $\pi_j(X/G)$. Furthermore, by~(\ref{E2}),
there is a homotopy equivalence $h: \Omega(X/G)\stackrel{\simeq}{\longrightarrow} G\times \Omega(X)$ that may be described in terms of a homotopy
between the inclusion map  $GQ_1\lhra X$ and the constant map. If $e\in G$ denotes the unit element of $G$, the restriction of the inverse of
the homotopy equivalence $h: \Omega(X/G)\stackrel{\simeq}{\longrightarrow} G\times \Omega(X)$ to $\{e\}\times \Omega(X)$ is homotopic to the loop
of the projection map $X \longrightarrow X/G$.

For the purpose of the following theorem, which is the main result of this section, for all $j\geq 0$, we identify $\pi_{j+1}(X/G)$ with $\pi_j(G)\times\pi_{j+1}(X)$ via the homotopy equivalence $h: \Omega(X/G)\stackrel{\simeq}{\longrightarrow} G\times \Omega(X)$, and $\pi_j(\prod_1^nX/G)$ with $\pi_j(\prod_1^{n-1}X/G)\times\pi_j(X/G)$. For this particular identification, we are taking $i=n$ in~(\ref{E'}).

\begin{thmX}\label{l}
Let $G\times X\longrightarrow X$ be a free, tame action of a Lie group $G$ on a connected topological manifold $X$ without boundary, and
let $Q_{1}\in X$.. Suppose that the inclusion map $i''_{n-1}(X): F_{n-1}^{G}(X\backslash GQ_1)\lhra\prod_1^{n-1}X$ is null homotopic
(this is the case if for example the inclusion map $X\backslash GQ_1\lhra X$ is null homotopic).

\begin{enumerate}[(1)]
\item\label{item:l1} The maps $\mathcal{I}(\mbox{\em id}_{F_{n-1}^G(X\backslash GQ_1)}, \mbox{\em id}_{\prod_1^{n-1}X}): I_{i_{n-1}''(X)}
\longrightarrow F_{n-1}^G(X\backslash GQ_1) \times \Omega(\prod_1^{n-1}X)$ and $\mathcal{I}(j_{n'}(X/G),j_{n}(X/G)): I_{i'_{n-1}(X/G)}\longrightarrow
I_{i_n(X/G)}$ are homotopy equivalences, and the map $\mathcal{I}(\psi_{n-1}(X \backslash GQ_1), \prod_1^{n-1}q(X)): I_{i_{n-1}''(X)}\longrightarrow
I_{i'_{n-1}(X/G)}$ is a weak homotopy equivalence.  Furthermore, if $G$ is discrete, $X$ is
$1$-connected and $\dim(X/G)\ge 3$ then the map $\psi_{n-1}(X\backslash GQ_1): F^G_{n-1}(X\backslash GQ_1)\longrightarrow F_{n-1}((X/G)\backslash \bar Q_1)$ is the universal covering.
\item\label{it:l2} If $j\leq \min(\operatorname{conn}(X), \dim (X/G)-2)$ then $\pi_j(F_{n-1}^G(X\backslash GQ_1))=0$.
\item\label{it:l3} Suppose that the homomorphism $\pi_j(X)\longrightarrow\pi_j(X/G)$ induced by the quotient map $X\longrightarrow X/G$ is injective
for all $j\geq 1$ (this is the case if for example the inclusion map $GQ_1\lhra X$ is null homotopic).
For all $j\geq 2$, up to the identification of the groups $\pi_{j-1}(I_{i_n(X/G)})$ and $\pi_{j-1}(F^G_{n-1}(X\backslash GQ_1))\times\pi_{j-1}(\Omega(\prod_1^{n-1}X))$ via the isomorphism
\begin{equation*}
\textstyle\pi_{j-1}(\mathcal{I}(\mbox{\em id}_{F_{n-1}^G(X\backslash GQ_1)}, \mbox{\em id}_{\prod_1^{n-1}X})) \circ (\pi_{j-1}(k_{n}(X)))^{-1},
\end{equation*}
the restriction to the subgroup $\pi_{j}(\prod_1^{n-1}X)$ of $\pi_{j}(\prod_1^{n}X/G)$ of the boundary homomorphism $\hat\partial_{j}: \pi_{j-1}(\Omega(\prod_1^nX/G)) \longrightarrow \pi_{j-1}(I_{i_n(X/G)})$ given by the long exact sequence in homotopy of the homotopy fibration $I_{i_n(X/G)}\xrightarrow{p_{i_n(X/G)}} F_n(X/G)\stackrel{i_n(X/G)}{\lhra} \prod_1^nX/G$
of $(\ref{threefibres})$ coincides with the inclusion of $\pi_{j}(\prod_1^{n-1}X)$ in
$\pi_{j-1}(F^G_{n-1}(X\backslash GQ_1))\times\pi_{j-1}(\Omega(\prod_1^{n-1}X))$ via the usual identification of $\pi_{j}(X)$ with
$\pi_{j-1}(\Omega(X))$.
\end{enumerate}
\end{thmX}

\begin{proof}\mbox{}
\begin{enumerate}[(1)]
\item By hypothesis, $i_{n-1}''(X)$ is null homotopic. Let $H: F_{n-1}^G(X\backslash GQ_1) \times I \longrightarrow \prod_1^{n-1}X$ be a homotopy
between $i_{n-1}''(X)$ and the constant map at the basepoint of $\prod_1^{n-1}X$, and let $\varphi(x): I \longrightarrow (\prod_1^{n-1}X)^{I}$
be the path given by $\varphi(x)(t)=H(x,t)$. Applying Remark~\ref{defIgg}(\ref{it:deflggd}) to the homotopy fibration
\begin{equation}\label{eq:homofibreidoubleprime}
I_{i_{n-1}''(X)}\xrightarrow{p_{i''_{n-1}(X)}} F_{n-1}^G(X\backslash GQ_1) \xrightarrow{i''_{n-1}(X)} \prod_1^{n-1} X,
\end{equation}
we obtain a homotopy equivalence:
\begin{equation}\label{eq:Iiprimeprime}
\textstyle\mathcal{I}(\text{\tiny id}_{F_{n-1}^G(X\backslash GQ_1)}, \text{\tiny id}_{\prod_1^{n-1}X}): I_{i_{n-1}''(X)} \longrightarrow F_{n-1}^G(X\backslash GQ_1) \times \Omega(\prod_1^{n-1}X)
\end{equation}
defined by $\mathcal{I}(\text{\tiny id}_{F_{n-1}^G(X\backslash GQ_1)}, \text{\tiny id}_{\prod_1^{n-1}X})(x,\gamma)=(x, \varphi(x)^{-1}\ast \gamma)$
for all $(x,\gamma)\in I_{i_{n-1}''(X)}$.
Since the map $\psi_{n-1}(X \backslash GQ_1) : F_{n-1}^G(X \backslash GQ_1)\longrightarrow F_{n-1}((X/G)\backslash\bar{Q}_1)$ is $1$-connected,
the map $\mathcal{I}(\psi_{n-1}(X \backslash GQ_1), \prod_1^{n-1}q(X)): I_{i_{n-1}''(X)}
\longrightarrow I_{i'_{n-1}(X/G)}$ is a weak homotopy equivalence by Corollary~\ref{LL1} and~(\ref{E-11}). Further, by~(\ref{E'-11}), the map $\mathcal{I}(j_{n'}(X/G),j_{n}(X/G)): I_{i'_{n-1}(X/G)}\longrightarrow I_{i_n(X/G)}$ is a homotopy equivalence, and this proves the first part of the statement.

Assume that $G$ is a discrete group. Since $X$ is a $1$-connected manifold, the quotient map
$\prod_1^{n-1}q(X) :\prod_{1}^{n-1} X\longrightarrow \prod_{1}^{n-1} X/G$ is the universal covering. Now $\dim(X/G)\geq 3$, and applying
Corollary~\ref{cc}, we see that the homomorphism $\pi_1(i'_{n-1}(X/G)) : \pi_1(F_{n-1}((X/G)\backslash \bar{Q}_1))\longrightarrow\pi_1(\prod_1^{n-1} X/G)$
is an isomorphism. It follows from diagram~(\ref{E}) and Lemma~\ref{L} that $\psi_{n-1}(X\backslash GQ_1) : F_{n-1}^G(X\backslash GQ_1)\longrightarrow F_{n-1}((X/G)\backslash \bar Q_1)$
is the universal covering.

\item Let $j\leq \min(\operatorname{conn}(X), \dim(X/G)-2)$. Applying the long exact homotopy sequence to the following commutative diagram whose rows are fibrations:
\begin{equation}\label{eq:connx}
\begin{gathered}
\xymatrix@1{\prod_1^{n-1}G \ar@{=}[d]\ar[r] & F_{n-1}^G(X\backslash GQ_1) \ar[rrr]^-{\psi_{n-1}(X \backslash GQ_1)}\ar@{^{(}->}[d]_-{i''_{n-1}(X)} & &&F_{n-1}((X/G)\backslash\bar{Q}_1)\ar@{^{(}->}[d]^{i'_{n-1}(X/G)}\\
\prod_1^{n-1}G\ar[r] & \prod_1^{n-1}X\ar[rrr]^{\prod_1^{n-1}q(X)} & &&\prod_1^{n-1}X/G}
\end{gathered}
\end{equation}
yields the following commutative diagram whose rows are exact:
\begin{equation}\label{eq:les}
\begin{gathered}
\xymatrix@C-=0.51cm@1{\pi_{j}(\prod_1^{n-1}G) \ar[r] \ar@{=}[d] & \pi_{j}(F_{n-1}^G(X\backslash GQ_1)) \ar[r]^*+<0.8em>-{\scriptstyle\pi_{j}(\psi_{n-1}(X \backslash GQ_1))} \ar[d]_-{\pi_{j}(i''_{n-1}(X))} & \pi_{j}(F_{n-1}((X/G)\backslash\bar{Q}_1)) \ar[r]^-{\partial_{j}} \ar[d]^-{\pi_{j}(i'_{n-1}(X/G))} & \pi_{j-1}(\prod_1^{n-1}G) \ar@{=}[d]\\
\pi_{j}(\prod_1^{n-1}G) \ar[r]&\pi_{j}(\prod_1^{n-1} X) \ar[r]^*+<0.5em>-{\scriptstyle\pi_{j}(\prod_1^{n-1}q(X))} & \pi_{j}(\prod_1^{n-1}X/G) \ar[r]_-{\bar{\partial}_j} & \pi_{j-1}(\prod_1^{n-1} G),}
\end{gathered}
\end{equation}
where $\partial_{j}: \pi_{j}(F_{n-1}((X/G)\backslash\bar{Q}_1)) \longrightarrow \pi_{j-1}(\prod_1^{n-1}G)$ (resp.\ $\bar\partial_{j}: \pi_{j}(\prod_1^{n-1}X/G) \longrightarrow \pi_{j-1}(\prod_1^{n-1} G)$) is the boundary homomorphism corresponding to the first (resp.\ second) row of~(\ref{eq:connx}). Since $j\leq \dim (X/G)-2$, by applying Corollary~\ref{cc} to the manifold $X/G$, we see that the homomorphism $\pi_{k}(i'_{n-1}(X/G))$ is surjective for all $k\leq j+1$ and is an isomorphism if $k\leq j$. The fact that $j\leq \operatorname{conn} (X)$ implies  that the homomorphism $\pi_{k}(\prod_1^{n-1}X/G) \longrightarrow \pi_{k-1}(\prod_1^{n-1} G)$ is surjective for all $k\leq j+1$ and is an isomorphism if $k\leq j$. By the commutativity of the diagram~(\ref{eq:les}), it follows that the boundary homomorphism $\partial_{k} :\pi_{k}(F_{n-1}((X/G)\backslash\bar{Q}_1)) \longrightarrow \pi_{k-1}(\prod_1^{n-1}G)$ is surjective for all $k\leq j+1$ and is an isomorphism if $k\leq j$, and so by exactness of the uppermost row of~(\ref{eq:les}), $\pi_{j}(F_{n-1}^G(X\backslash GQ_1))=0$ for all $j\leq \min(\operatorname{conn}(X), \dim(X/G)-2)$.

\item Let $\tilde{d}: \Omega(\prod_1^{n-1}X)\longrightarrow I_{i_{n-1}''(X)}$ denote the boundary map of the homotopy
fibration~(\ref{eq:homofibreidoubleprime}), and let $\tilde\partial_j: \pi_j(\prod_1^{n-1}X)\longrightarrow \pi_{j-1}(F^G_{n-1}(X\backslash GQ_1))\times\pi_{j-1}(\Omega(\prod_1^{n-1}X))$
denote the composition of the associated boundary homomorphism $\pi_{j-1}(\tilde{d})$ with the isomorphism
$\pi_{j-1}(\mathcal{I}(\text{\tiny id}_{F_{n-1}^G(X\backslash GQ_1)}, \text{\tiny id}_{\prod_1^{n-1}X}))$ induced by the homotopy equivalence
given by~(\ref{eq:Iiprimeprime}). Identifying $\pi_{j}(X)$ with $\pi_{j-1}(\Omega(X))$, and applying Remark~\ref{defIgg}(\ref{it:deflggd}), it follows that
$\tilde\partial_j$ coincides with the inclusion homomorphism $$\pi_j\left(\prod_1^{n-1}X\right)\lhra\pi_{j-1}(F^G_{n-1}(X\backslash GQ_1))\times\pi_{j-1}\left(\Omega\left(\prod_1^{n-1}X\right)\right).$$
The first and third rows of diagram~(\ref{threefibres}) give rise to the following commutative diagram of homotopy fibrations:
\begin{equation}\label{eq:cdhomfib1}
\begin{gathered}
\xymatrix{I_{i_{n-1}''(X)}\ar[rr]^-{p_{i_{n-1}''(X)}}\ar[d]^{k_{n}(X)} && F_{n-1}^G(X\backslash GQ_1)\ar[d]^{j_{n}'(X/G)\circ \psi_{n-1}(X\backslash GQ_1)}\ar@{^{(}->}[rr]^-{i''_{n-1}(X)}&& \prod_1^{n-1}X\ar[d]^{j_{n}(X/G)\circ \prod_1^{n-1}q(X)}\\
I_{i_n(X/G)}\ar[rr]_-{p_{i_n(X/G)}} &&F_n(X/G)\ar@{^{(}->}[rr]_ {i_n(X/G)}&&\prod_1^nX/G,}
\end{gathered}
\end{equation}
where the leftmost vertical map $k_{n}(X): I_{i_{n-1}''(X)} \longrightarrow I_{i_n(X/G)}$ is defined by $k_{n}(X)= \mathcal{I}(j_{n}'(X/G), j_{n}(X/G))\circ \mathcal{I}(\psi_{n-1}(X \backslash GQ_1),\Pi_1^{n-1}q(X))$, and is a weak homotopy equivalence by part~(\ref{item:l1}). Now let $\hat{d}: \Omega(\prod_1^nX/G) \longrightarrow I_{i_n(X/G)}$ be the boundary map of the lower homotopy fibration of~(\ref{eq:cdhomfib1}), and let $$\hat\partial_{j}: \pi_{j-1}(\Omega(\prod_1^nX/G)) \longrightarrow \pi_{j-1}(I_{i_n(X/G)})$$ be the induced boundary homomorphism. 

Consider the element $\omega=(\omega_{1}, \ldots, \omega_{n-1}, c_{q_{n-1}}) \in \Omega(\prod_1^{n-1}X)$. The map $\prod_{1}^{n} q(X)$ induces a map $\Omega \prod_{1}^{n} q(X): \Omega (\prod_{1}^{n} X) \longrightarrow \Omega (\prod_{1}^{n} X/G)$, so
\begin{align}
\textstyle(\Omega \prod_{1}^{n} q(X))(\omega,c_{q_{n}})&= (\bar\omega_{1}, \ldots, \bar\omega_{n-1}, c_{\bar{q}_{n}}) \in \textstyle\Omega(\prod_1^nX/G )\; \text{and} \notag\\
\textstyle\hat{d} \circ (\Omega \prod_{1}^{n} q(X))(\omega)&=(\bar{q}_{1},\ldots, \bar{q}_{n}, \bar\omega_{1}, \ldots, \bar\omega_{n-1}, c_{\bar{q}_{n}})\in I_{i_{n}(X/G)}\label{eq:dhatomega}
\end{align}
by Remark~\ref{defIgg}(\ref{it:deflggd}). On the other hand, $(q_{1}, \ldots, q_{n-1}, \omega_{1}, \ldots,\omega_{n-1})\in I_{i_{n-1}''(X)}$, and by Remark~\ref{defIgg}(\ref{it:deflggb}), the definition of $k_{n}(X)$,~(\ref{eq:Iiprimeprime}) and~(\ref{eq:dhatomega}), we have:
\begin{gather*}
k_{n}(X)(q_{1}, \ldots, q_{n-1}, \omega_{1}, \ldots,\omega_{n-1})=\textstyle\hat{d} \circ (\Omega \prod_{1}^{n} q(X))(\omega)\;\text{and}\\
\mathcal{I}(\text{\tiny id}_{F_{n-1}^G(X\backslash GQ_1)}, \text{\tiny id}_{\prod_1^{n-1}X})(q_{1}, \ldots, q_{n-1}, \omega_{1}, \ldots,\omega_{n-1})= (q_{1}, \ldots, q_{n-1}, \omega_{1}, \ldots,\omega_{n-1}).
\end{gather*}
Thus the image of $([\omega], [c_{q_{n}}])\in \pi_j(\prod_1^{n-1}X)\times\pi_j(X)$ under the composition
\begin{equation*}
\textstyle\pi_{j-1}(\mathcal{I}(\text{\tiny id}_{F_{n-1}^G(X\backslash GQ_1)}, \text{\tiny id}_{\prod_1^{n-1}X})) \circ (\pi_{j-1}(k_{n}(X)))^{-1}\circ \hat\partial_{j}\circ \pi_{j}(\prod_1^n q(X))
\end{equation*}
is equal to $([c_{q_{1}}],\ldots, [c_{q_{n-1}}], [\omega])$. In particular, up to identification of $\pi_{j-1}(I_{i_n(X/G)})$ and $\pi_{j-1}(F^G_{n-1}(X\backslash GQ_1))\times\pi_{j-1}(\Omega(\prod_1^{n-1}X))$ via the  isomorphism
\begin{equation*}
\textstyle\pi_{j-1}(\mathcal{I}(\text{\tiny id}_{F_{n-1}^G(X\backslash GQ_1)}, \text{\tiny id}_{\prod_1^{n-1}X})) \circ (\pi_{j-1}(k_{n}(X)))^{-1},
\end{equation*}
the restriction of the homomorphism $\hat\partial_{j}: \pi_{j-1}(\Omega(\prod_1^nX/G)) \longrightarrow \pi_{j-1}(I_{i_n(X/G)})$ to the subgroup $\pi_{j}(\prod_1^{n-1}X)$ of $\pi_{j}(\prod_1^{n}X/G)$ coincides with the inclusion of $\pi_{j}(\prod_1^{n-1}X)$ in $\pi_{j-1}(F^G_{n-1}(X\backslash GQ_1))\times\pi_{j-1}(\Omega(\prod_1^{n-1}X))$
which completes the proof.\qedhere
\end{enumerate}

\end{proof}

\medskip

Theorem~\ref{l} implies the following corollary.

\begin{CorX}\label{Corr} Let $X$ be a connected topological manifold without boundary such that $\operatorname{\text{dim}}(X)\geq 3$, let $\tilde{X}$ be its universal covering, and let $Q_{1}\in \tilde{X}$.
Suppose that the map $i''_{n-1}(\tilde{X}) : F_{n-1}^{\pi_1(X)}(\tilde{X}\backslash\pi_1(X)Q_1)\lhra\prod_1^{n-1}\tilde{X}$ is null homotopic.
\begin{enumerate}[(1)]
\item The spaces $F_{n-1}^{\pi_1(X)}(\tilde{X}\backslash\pi_1(X)Q_1)$ and $\widetilde{F_{n-1}(X\backslash \bar{Q}_1)}$ are homeomorphic,
and there exists a homotopy equivalence $I_{i''_{n-1}(\tilde{X})}\simeq \widetilde{F_{n-1}(X\backslash \bar{Q}_1)}\times\Omega(\prod_1^{n-1}\tilde{X})$,
and a weak homotopy equivalence between $I_{i_{n}(X)}$ and $\widetilde{F_{n-1}(X\backslash \bar{Q}_1)}\times\Omega(\prod_1^{n-1}\tilde{X})$.

\item Let $j\geq 2$. Up to the identification of the groups
$\pi_{j}(\prod_1^n\tilde{X})$ with $\pi_{j}(\prod_1^nX)$ via the
isomorphism $\pi_{j}(q(\tilde{X})): \pi_{j}(\tilde{X}) \longrightarrow
\pi_{j}(X)$ and the identification of the groups
$\pi_{j-1}(I_{i_n(X)})$ and
$\pi_{j-1}(F^{\pi_1(X)}_{n-1}(\tilde{X}\backslash
\pi_1(X)Q_1))\times\pi_{j-1}(\Omega(\prod_1^{n-1}\tilde{X}))$ via the
isomorphism given in the statement of {\em Theorem~\ref{l}(\ref{it:l3})},
the restriction of the boundary homomorphism $\hat\partial_{j}:
\pi_{j}(\prod_1^n \tilde{X}) \longrightarrow \pi_{j-1}(I_{i_n(X)})$ of
the homotopy
long exact sequence of the homotopy fibration
$$I_{i_n(X)}\xrightarrow{p_{i_n(X)}} F_n(X)\stackrel{i_n(X)}{\lhra}
\prod_1^nX$$ to the subgroup $\pi_{j}(\prod_1^{n-1}\tilde{X})$ of
$\pi_{j}(\prod_1^{n}\tilde{X})$ coincides with the inclusion of
$\pi_{j}(\prod_1^{n-1}\tilde{X})$ in
$\pi_{j-1}(F^{\pi_1(X)}_{n-1}(\tilde{X}\backslash
\pi_1(X)Q_1))\times\pi_{j-1}(\Omega(\prod_1^{n-1}\tilde{X}))$ via the
usual identification of $\pi_{j}(\tilde{X})$ with
$\pi_{j-1}(\Omega(\tilde{X}))$.
\end{enumerate}
\end{CorX}

\begin{proof}
In the whole of the proof, we shall replace the manifold $X$ (resp.\ the group $G$) of Theorem~\ref{l} by $\tilde{X}$ (resp.\ by $\pi_{1}(X)$).
The group $\pi_1(X)$ is discrete, and since $X$ is a manifold, $\pi_1(X)$ is countable and so is a $0$-dimensional Lie group.
\begin{enumerate}[(1)]
\item Since $\pi_{1}(X)$ is discrete, $\tilde{X}$ is $1$-connected and $\dim(X)\geq 3$, it follows from the second part of the statement of
Theorem~\ref{l}(\ref{item:l1}) that $F_{n-1}^{\pi_1(X)}(\tilde{X}\backslash\pi_1(X)Q_1)$ is homeomorphic to $\widetilde{F_{n-1}(X\backslash \bar{Q}_1)}$. Applying this to the first part of Theorem~\ref{l}(\ref{item:l1}) yields the rest of the statement.

\item Since $q(\tilde{X}):\tilde{X}\longrightarrow X$ is a covering
map, the induced homomorphism $\pi_{j}(q(\tilde{X})):
\pi_j(\tilde{X})\longrightarrow\pi_j(X)$ is injective for all $j\geq
1$ and is an isomorphism for all $j\geq 2$. Up to the identification
of $\pi_{j}(\prod_1^n\tilde{X})$ with $\pi_{j}(\prod_1^nX)$ via this
isomorphism, the result then follows directly from
Theorem~\ref{l}(\ref{it:l3}).\qedhere
\end{enumerate}
\end{proof}

\medskip

The following result brings together various descriptions for the homotopy type of the homotopy fibre $I_{i_n(X)}$, where $X$ is a topological manifold without boundary whose universal covering is contractible.

\begin{proX}\label{AS}
Let $X$ be a connected topological manifold without boundary such that $\operatorname{\text{dim}}(X)\geq 3$ and whose universal covering $\tilde{X}$ is
contractible, and let $Q_1\in \tilde{X}$. Then the homotopy fibre $I_{i_n(X)}$ of the map $i_{n}(X):F_{n}(X) \longrightarrow \prod_1^nX$ is weakly homotopy equivalent to each of the following six spaces: $\widetilde{F_{n-1}(X \backslash \bar Q_1)}$, $F_n(X)\times_{\prod_1^nX}\prod_1^n \tilde X$, $\widetilde{F_n(X)}$, $F_n^{\pi_1(X)}(\tilde X)$,  $F_{n-1}(X\backslash \bar Q_1)\times_{\prod_1^{n-1}X}\prod_1^{n-1} \tilde X$ and $F_{n-1}^{\pi_1(X)}(\tilde X\backslash \pi_1(X)Q_1)$.
\end{proX}
\begin{proof}
Let $p(\tilde X) : \tilde X\longrightarrow X$ denote the universal covering map.
\begin{enumerate}[\textbullet]
\item Since the space $\prod_1^n \tilde X$ is contractible, the map $i''_{n-1}(\tilde{X}): F_{n-1}^{\pi_{1}(X)}(\tilde{X}\backslash \pi_{1}(X) Q_1)\lhra\prod_1^{n-1}\tilde{X}$ is null homotopic. So by~(\ref{E1}), $I_{i''_{n-1}(\tilde{X})}\simeq F_{n-1}^{\pi_{1}(X)}(\tilde{X}\backslash  \pi_{1}(X) Q_1)$. Further, $\pi_{1}(X)$ is discrete, $\tilde{X}$ is $1$-connected and $\dim(X)\geq 3$, it follows from Theorem~\ref{l}(\ref{item:l1}) that $F^{\pi_{1}(X)}_{n-1}(\tilde{X}\backslash  \pi_{1}(X) Q_1)\simeq \widetilde{F_{n-1}(X\backslash \bar Q_1)}$.

\item The strict pull-back~(\ref{E}) gives rise to a homotopy equivalence $F_{n-1}^{\pi_1(X)}(\tilde X\backslash \pi_1(X)Q_1) \simeq F_{n-1}(X\backslash \bar Q _1)\times_{\prod_1^{n-1}X}\prod_1^{n-1} \tilde X$.

\end{enumerate}
So $I_{i''_{n-1}(\tilde{X})}$, $F^{\pi_{1}(X)}_{n-1}(\tilde{X}\backslash  \pi_{1}(X) Q_1)$, $\widetilde{F_{n-1}(X\backslash \bar Q_1)}$ and $F_{n-1}(X\backslash \bar Q _1)\times_{\prod_1^{n-1}X}\prod_1^{n-1} \tilde X$ are pairwise homotopy equivalent.

\begin{enumerate}[\textbullet]
\item Since $\dim(X)\geq 3$, the map $i_{n}(X):F_{n}(X) \longrightarrow \prod_1^nX$ is $2$-connected by Theorem~\ref{B}, and thus $\pi_{1}(i_{n}(X))$ is an isomorphism. Further, $\prod_1^n p(\tilde X): \prod_1^n \tilde X \longrightarrow \prod_1^nX$ is the universal covering, and so it follows from Lemma~\ref{L} that $\widetilde{F_{n}(X)} \simeq F_n(X)\times_{\prod_1^nX}\prod_1^n\tilde X$.


\item The pull-back~(\ref{pullbackFGnX}) gives rise to a homotopy equivalence $F_n(X)\times_{\prod_1^nX}\prod_1^n \tilde X\simeq F_n^{\pi_1(X)}(\tilde{X})$.

\item Consider the (strict) pull-back:
$$\xymatrix{F_n(X)\times_{\prod_1^nX}\prod_1^n\tilde X\ar[d]_{p'(X)}\ar[rr]^-{\tilde{i}_n(X)}&& \prod_1^n \tilde X\ar[d]^{\prod_1^n p(\tilde X)}\\
F_n(X)\ar@{^{(}->}[rr]^-{i_n(X)}&& \prod_1^nX.}$$
Since the projection map $p'(X) : F_n(X)\times_{\prod_1^nX}\prod_1^n\tilde X\longrightarrow F_n(X)$ is $1$-connected,
the map $\mathcal{I}(p'(X), \prod_1^np(\tilde{X})):  I_{\tilde{i}_n(X)}\longrightarrow I_{i_{n}(X)}$ is a weak homotopy equivalence by Corollary~\ref{LL1} and~(\ref{E-11}).
On the other hand, the map $\tilde{i}_n(X)$ is null homotopic because the space $\prod_1^n \tilde X$ is contractible, and it follows from~(\ref{E1}) that $I_{\tilde{i}_n(X)} \simeq F_n(X)\times_{\prod_1^nX}\prod_1^n\tilde X$.
\end{enumerate}
Thus $I_{\tilde{i}_n(X)}$, $F_n(X)\times_{\prod_1^nX}\prod_1^n\tilde X$, $F_n(X)\times_{\prod_1^nX}\prod_1^n \tilde X$ and $F_n^{\pi_1(X)}(\tilde{X})$ are pairwise homotopy equivalent, and are weakly homotopy equivalent to $I_{i_{n}(X)}$. Finally,~(\ref{E'-11}) implies that $I_{i_{n-1}'(X)}\simeq I_{i_{n}(X)}$, and by Corollary~\ref{LL1} and~(\ref{E-11}), there is a weak homotopy equivalence $I_{i''_{n-1}(\tilde{X})}\longrightarrow I_{i'_{n-1}(X)}$, and the result follows.\end{proof}

The family of manifolds covered by Euclidean spaces is properly contained in the family of manifolds whose universal covering is
contractible. The fact that the inclusion is strict may be seen by considering for example the Whitehead open $3$-manifold that is
contractible but not homeomorphic to $\mathbb{R}^3$.


\begin{RemX}
By \cite[Theorem~1.6, Chapter~V, p.\ 229]{Br} and~\cite{Zb},  the universal covering of any paracompact aspherical manifold
is contractible.
\end{RemX}



\setcounter{section}{3}
\setcounter{equation}{0}
\setcounter{thmX}{0}

\noindent
{\large\bf 3.\ The homotopy fibre of the inclusion map $F_n(\mathbb{S}^k/G)\lhra\prod_1^n\mathbb{S}^k/G$ and the associated long exact sequence in homotopy}

\medskip

In this section, we start by studying the homotopy fibre of the inclusion  $i_n(\mathbb{S}^k/G) : F_n(\mathbb{S}^k/G)\lhra\prod_1^n\mathbb{S}^k/G$, where $\mathbb{S}^k/G$ is the orbit manifold associated with a free tame action of a Lie group $G$ on the $k$-sphere $\mathbb{S}^k$.
One of our aims is to describe completely the long exact sequence in homotopy of the homotopy fibration
$$I_{\iota_{n}(\mathbb{S}^k/G)}\longrightarrow F_n(\mathbb{S}^k/G) \longrightarrow \prod_1^n\mathbb{S}^k/G$$
in the case that $G$ is discrete (so finite).
Let $k\geq 3$, and consider a free tame action $G\times\mathbb{S}^k\longrightarrow \mathbb{S}^k$ of a compact Lie group $G$ on $\mathbb{S}^k$. If $\dim(G)=0$ then $G$ is finite, and comparing Euler characteristics, we obtain $\chi(\mathbb{S}^k)=\left\lvert G \right\rvert \ldotp \chi(\mathbb{S}^k/G)$, so $G$ is trivial or $\Z_{2}$ if $k$ is even, and $\left\lvert G \right\rvert\ge 2$ if $k$ is odd. In particular, for $G=\Z_2$ we get $\S^k/\Z_2=\R P^k$, the $k$-dimensional real projective space. If $\dim(G)>0$ then we have the following result.

\begin{thmX}[{\cite[8.5.~Theorem]{Br}}]\label{BB}
Suppose that $G$ is a compact Lie group such that $\dim(G)>0$ and that acts freely on $\mathbb{S}^k$. Then $G$ is isomorphic to a subgroup of $\mathbb{S}^3$. More precisely, up to isomorphism, $G$ is one of the following groups: $\mathbb{S}^1$; $N(\mathbb{S}^1)$, the normaliser of $\mathbb{S}^1$ in the group $\mathbb{S}^3$; or $\mathbb{S}^3$.
\end{thmX}

Using Theorem~\ref{BB}, we obtain the following proposition.

\begin{proX}
Let $G$ be a compact Lie group $G$, and let $G\times \mathbb{S}^k\longrightarrow \mathbb{S}^k$ be a free action of $G$ on the sphere $\mathbb{S}^k$.
\begin{enumerate}[(1)]
\item\label{it:faSk} If $G$ is finite with $\left\lvert G \right\rvert>2$ or $G = \mathbb{S}^1$ then $k$ is odd.
\item If $G = N(\mathbb{S}^1)$ or $G = \mathbb{S}^3$ then $k=4m+3$ for some $m\ge 0$.

\end{enumerate}
\end{proX}

\begin{proof}\mbox{}
\begin{enumerate}[(1)]
\item First, suppose that $G$ is finite with $\left\lvert G \right\rvert>2$. Given a free action $G\times\mathbb{S}^k\longrightarrow \mathbb{S}^k$,
we see that $\chi(\mathbb{S}^k)=\left\lvert G \right\rvert\ldotp\chi(\mathbb{S}^k/G)$, so $k$ is odd because $\left\lvert G \right\rvert>2$. Now assume that $G = \mathbb{S}^1$. A free action $\mathbb{S}^1\times\mathbb{S}^k\longrightarrow \mathbb{S}^k$ of $\mathbb{S}^1$ on $\mathbb{S}^k$ induces a free action of the cyclic group $\mathbb{Z}_l$ on $\mathbb{S}^k$ for all $l\geq 1$. Comparing Euler characteristics once more, we obtain $\chi(\mathbb{S}^k)=l \ldotp \chi(\mathbb{S}^k/\mathbb{Z}_l)$, and taking $l\ge 3$, we deduce once more that $k$ is odd.

\item Suppose that $G = N(\mathbb{S}^1)$ or $\mathbb{S}^3$. A free action $\mathbb{S}^3\times\mathbb{S}^k\longrightarrow \mathbb{S}^k$ of $\mathbb{S}^3$ on $\mathbb{S}^k$ induces a free action $N(\mathbb{S}^1)\times\mathbb{S}^k\longrightarrow \mathbb{S}^k$ of $N(\mathbb{S}^1)$ on $\mathbb{S}^k$, and a free action $N(\mathbb{S}^1)\times\mathbb{S}^k\longrightarrow \mathbb{S}^k$ of $G$ on $\mathbb{S}^k$ induces a free action $\mathbb{S}^1\times\mathbb{S}^k\longrightarrow \mathbb{S}^k$ of $\mathbb{S}^1$ on $\mathbb{S}^k$. By~(\ref{it:faSk}), it follows that $k=2m'+1$ for some $m'\ge 0$, from which we obtain an induced free action $\mathbb{Z}_2\times \mathbb{S}^{2m'+1}/\mathbb{S}^1\longrightarrow \mathbb{S}^{2m'+1}/\mathbb{S}^1$ of $\mathbb{Z}_2$ on $\mathbb{S}^{2m'+1}/\mathbb{S}^1$. Using the Euler characteristic relation $\chi(\mathbb{S}^{2m'+1}/\mathbb{S}^1) + \chi((\mathbb{S}^{2m'+1}/\mathbb{S}^1)^{\mathbb{Z}_2}) = 2\tilde{\chi}((\mathbb{S}^{2m'+1}/\mathbb{S}^1)/\mathbb{Z}_2)$ given in~\cite[p.~146]{Fl} for the reduced Euler characteristic $\tilde{\chi}$, we see that $m'$ is also odd, and so $k=4m+3$ for some $m\ge 0$.\qedhere
\end{enumerate}
\end{proof}

If $k\geq 3$, let $G\times \mathbb{S}^k\longrightarrow\mathbb{S}^k$ be a tame, free action of a compact Lie group $G$ on $\mathbb{S}^k$, and consider the topological manifold without boundary $\mathbb{S}^k/G$ that is the orbit space for this action. If $Q_{1}\in \mathbb{S}^k$, the map $\mathbb{S}^k\backslash GQ_1\lhra\mathbb{S}^k$ factors through the contractible space $\mathbb{S}^k\backslash Q_1$, and so is null homotopic. If $1\le i\le n$, the fact that the following diagram:
$$\xymatrix{F_n^G(\mathbb{S}^k\backslash GQ_1)\ar[d]_{J_i(\mathbb{S}^k)\left\lvert_{F_n^G(\mathbb{S}^k/GQ_{1})}\right.} \ar@{^{(}->}[rr]^-{i''_n(\mathbb{S}^k)}&& \prod_1^n\mathbb{S}^k\ar[d]^{J_i(\mathbb{S}^k)}\\
\mathbb{S}^k\backslash GQ_1\ar@{^{(}->}[rr] && \mathbb{S}^k}$$
is commutative implies that the map $i_n''(\mathbb{S}^k) : F_n^G(\mathbb{S}^k\backslash GQ_1)\lhra\prod_1^n\mathbb{S}^k$ is also null homotopic. We obtain the following corollary by applying parts~(\ref{item:l1}) and~(\ref{it:l2}) of Theorem~\ref{l} to the map $i_n''(\mathbb{S}^k)$.

\begin{CorX} \label{P}
Let $G\times \mathbb{S}^k\longrightarrow \mathbb{S}^k$ be a free, tame action of a compact Lie group $G$ on $\mathbb{S}^k$.
\begin{enumerate}[(1)]
\item There is a homotopy equivalence $I_{i''_{n-1}(\mathbb{S}^k)} \simeq F_{n-1}^G(\mathbb{S}^k\backslash GQ_1)\times \Omega(\prod_1^{n-1}\mathbb{S}^k)$.
If the group $G$ is finite and $\dim(\mathbb{S}^k/G)\geq 3$ then $\widetilde{F_{n-1}(\mathbb{S}^k\backslash Q_1)} \simeq F_{n-1}^G(\mathbb{S}^k\backslash GQ_1)$.

\item\label{it:P2} If $j\leq k-\dim G-2$ then $\pi_j(F_{n-1}^G(\mathbb{S}^k\backslash GQ_1))=0$.
\end{enumerate}
\end{CorX}

Let $k\geq 3$. We now proceed to study the long exact sequence in homotopy of the homotopy fibration
 $$I_{i_n({\mathbb{S}^k/G})}\xrightarrow{p_{i_n({\mathbb{S}^k/G})}} F_n({\mathbb{S}^k/G})\xrightarrow{i_n({\mathbb{S}^k/G})}\prod_1^n\mathbb{S}^k/G.$$
 In this case, there is a homotopy equivalence $I_{i''_{n-1}(\mathbb{S}^k)}\stackrel{\simeq}{\longrightarrow} F_{n-1}^G(\mathbb{S}^k\backslash GQ_1)\times \Omega(\prod_1^{n-1}\mathbb{S}^k)$
 by Theorem~\ref{l}(\ref{item:l1}). We restrict our attention to the case where $k$ is odd and $\dim(G)=0$ \emph{i.e.}, $G$ is finite. Up to the identification
 of the group $\pi_{j}(I_{i_n({\mathbb{S}^k/G})})$ with $\pi_{j-1}(F_{n-1}^G(\mathbb{S}^k\backslash GQ_1)\times \Omega(\prod_1^{n-1}\mathbb{S}^k))$ described in Theorem~\ref{l}(\ref{it:l3}), our aim is to describe the boundary homomorphism:
 $$\hat{\partial}_j: \pi_j\biggl( \prod_1^{n}\mathbb{S}^k \biggr) \longrightarrow \pi_{j-1}\biggl(F_{n-1}^G(\mathbb{S}^k\backslash GQ_1)\times \Omega\biggl(\prod_1^{n-1}\mathbb{S}^k\biggr)\biggr),$$
as well as other homomorphisms of that exact sequence. Since $k\geq 3$, Corollary \ref{P}(\ref{it:P2}) implies that $\pi_1( I_{i_n(\mathbb{S}^k/G)})=0$. So we may assume that $j>2$.
In what follows, we shall need the following result (cf.~\cite[Theorem~3.1]{FH}).

\begin{lemX}\label{fad}\label{D}
Let $k\geq 3$ be odd, and let $G\times \mathbb{S}^k\longrightarrow\mathbb{S}^k$ be a free action of a finite group $G$ on $\mathbb{S}^k$.
Then for all $1\leq i,l\leq n$, there exists a map $\lambda_i(\mathbb{S}^k) : \mathbb{S}^k\longrightarrow F_n^G(\mathbb{S}^k)$ such that the composition $\mathbb{S}^k\xrightarrow{\lambda_i(\mathbb{S}^k)} F_n^G(\mathbb{S}^k)\xrightarrow{\psi_n(\mathbb{S}^k)} F_n(\mathbb{S}^k/G) \xrightarrow{p_{J_l}(\mathbb{S}^k/G)\left\lvert_{F_n(\mathbb{S}^k/G)}\right.} \mathbb{S}^k/G$ is homotopic to the quotient map $q(\mathbb{S}^k) :\mathbb{S}^k\longrightarrow\mathbb{S}^k/G$ for $1\le i\le n$. If $l=i$ then the composition coincides with $q(\mathbb{S}^k)$.
\end{lemX}
\begin{proof}
To construct the continuous map $\lambda_i(\mathbb{S}^k) : \mathbb{S}^k\longrightarrow F_n^G(\mathbb{S}^k)$, consider a non-vanishing vector
field on $\mathbb{S}^k$. Such a vector field exists because $k$ is odd. Since the given action $G\times \mathbb{S}^k\longrightarrow\mathbb{S}^k$
is free, there exists $\epsilon>0$ such that the distance between any two distinct points belonging to the same $G$-orbit is greater than $\epsilon$.
Given $x\in \mathbb{S}^k$, consider the $n$-tuple $(x, x_2,\ldots, x_n)$, where for $m=2,\ldots,n$, $x_{m}$  is the point on the integral curve of
the vector field passing through $x$ and corresponding to the parameter value $t=(m-1)/R$ for $R>0$ sufficiently large. Thus the distance between any
two points of this $n$-tuple is less than $\epsilon$, and so the points $x, x_2,\ldots, x_{n-1}$ and $x_n$ belong to distinct orbits. Given
$1\leq i\leq n$, set $\lambda_i(x)=(x_2,\ldots, x_{i},\underbrace{x}_{\text{$i$\textsuperscript{th} position}}, x_{i+1},\ldots x_n)$. By construction,
$\lambda_i(x)$ belongs to the orbit configuration space $F_n^G(\mathbb{S}^k)$ and $\lambda_i$ is continuous.
Further,
\begin{equation*}
p_{J_l}(\mathbb{S}^k/G)\left\lvert_{F_n(\mathbb{S}^k/G)}\right. \circ \psi_n(\mathbb{S}^k) \circ \lambda_i(x)= \begin{cases}
q(x_{i+1}) & \text{if $l<i$},\\
q(x) & \text{if $l=i$},\\
q(x_{i}) & \text{if $l>i$.}
\end{cases}
\end{equation*}
In each case, one may use the integral curve of the vector field to show that the composition $p_{J_l}(\mathbb{S}^k/G)\left\lvert_{F_n(\mathbb{S}^k/G)}\right. \circ \psi_n(\mathbb{S}^k) \circ \lambda_i: \mathbb{S}^k \longrightarrow\mathbb{S}^k/G$ is homotopic to $q(\mathbb{S}^k)$ as required.
\end{proof}

If the group $G$ is trivial, Lemma~\ref{fad} implies that the fibration $F_n(\mathbb{S}^k) \longrightarrow\mathbb{S}^k$ admits a section
(cf.~\cite[Theorem 3.1]{FH}). The lemma also yields some information about the homomorphism
$\pi_j(i_n(\mathbb{S}^k/G)) :\pi_j(F_n(\mathbb{S}^k)) \longrightarrow \pi_j(\prod_1^n\mathbb{S}^k/G)$ induced by the inclusion map
$i_n(\mathbb{S}^k/G): F_n(\mathbb{S}^k/G) \lhra \prod_1^n \mathbb{S}^k/G$, and which is relevant to understanding the long exact homotopy
sequence in homotopy of the homotopy fibration  $I_{i_n(\mathbb{S}^k/G)}\longrightarrow F_n(\mathbb{S}^k/G) \lhra \prod_1^n \mathbb{S}^k/G$.
In order to study the homomorphism $\pi_{j}(i_n(\mathbb{S}^k/G)):\pi_j(F_n(\mathbb{S}^k/G)) \longrightarrow \pi_j(\prod_1^n\mathbb{S}^k/G)$,
consider a basepoint $(x_1,\ldots,x_n)\in F_n(\mathbb{S}^k/G)$, where we assume that the set $\{x_1,\ldots,x_n\}$ is contained in a small disc
whose centre is a point $x_0\in \mathbb{S}^k/G$. If $1\le i\le n$, choosing a path from $x_i$ to $x_0$ inside this disc allows us to identify
$\pi_j(\mathbb{S}^k/G, x_i)$ with $\pi_j(\mathbb{S}^k/G,x_0)$, and in this way, the group $\pi_j(\prod_1^{n}{\mathbb{S}^k/G},(x_1,\ldots,x_n))$
may be identified isomorphically with the group $\prod_1^{n}\pi_j({\mathbb{S}^k/G}, x_0)$. In what follows, we shall consider the composition of
the homomorphism $\pi_{j}(i_n(\mathbb{S}^k/G)): \pi_j(F_n(\mathbb{S}^k/G)) \longrightarrow \pi_j(\prod_1^n\mathbb{S}^k/G, (x_1,\ldots,x_n))$ by
this isomorphism, which by abuse of notation, we shall also denote by $\pi_j(i_n(\mathbb{S}^k/G)): \pi_j(F_n(\mathbb{S}^k/G))
\longrightarrow \pi_j(\prod_1^{n}{\mathbb{S}^k/G})$.

\begin{lemX}\label{Lem}
Let $k\geq 3$ be odd, and let $G\times \mathbb{S}^k\longrightarrow\mathbb{S}^k$ be a free action of a finite group $G$ on $\mathbb{S}^k$. Then $\pi_0(F_n((\mathbb{S}^k/G)\backslash \bar Q_1))=0$, $\pi_1(F_n((\mathbb{S}^k/G)\backslash \bar Q_1)=\prod_1^nG$, and the homomorphism  $\pi_j(i_n'(\mathbb{S}^k/G)) : \pi_j(F_n((\mathbb{S}^k/G)\backslash \bar Q_1))\longrightarrow\pi_j(\prod_1^n\mathbb{S}^k/G)$ induced by the inclusion map $i_n'(\mathbb{S}^k/G) : F_n((\mathbb{S}^k/G)\backslash \bar Q_1)\lhra\prod_1^n\mathbb{S}^k/G$ is trivial for all $j\geq 2$.
\end{lemX}

\begin{proof} By Corollary \ref{cc}, the homomorphism $\pi_j(i_n'(\mathbb{S}^k/G)) : \pi_j(F_n((\mathbb{S}^k/G)\backslash \bar Q_1))\longrightarrow\pi_j(\prod_1^n\mathbb{S}^k/G)$ is an isomorphism for $j=0,1$,
whence $\pi_0(F_n((\mathbb{S}^k/G)\backslash \bar Q_1))=0$ and $\pi_1(F_n((\mathbb{S}^k/G)\backslash \bar Q_1))=\prod_1^nG$. So assume that $j\geq 2$. Taking $X=\mathbb{S}^k$ and $r=n=1$ in~(\ref{E}), we obtain the following pull-back diagram:
\begin{equation}\label{eq:psi1}
\begin{gathered}
\xymatrix{\mathbb{S}^k\backslash GQ_1\ar[d]_{\psi_1(\mathbb{S}^k \backslash GQ_1)}\ar@{^{(}->}[rr]^-{i''_{1}(\mathbb{S}^k)}&& \mathbb{S}^k\ar[d]^{q(\mathbb{S}^k)}\\
(\mathbb{S}^k/G)\backslash\bar{Q}_1\ar@{^{(}->}[rr]^-{i'_1( \mathbb{S}^k/G)}&& \mathbb{S}^k/G.}
\end{gathered}
\end{equation}
Now $\pi_1(i_1'(\mathbb{S}^k/G))$ is an isomorphism, and so $\psi_1(\mathbb{S}^k \backslash GQ_1): \mathbb{S}^k\backslash GQ_1 \longrightarrow (\mathbb{S}^k/G)\backslash\bar{Q}_1$ is the universal covering by Lemma~\ref{L}. In particular, $\pi_{j}(\psi_1(\mathbb{S}^k \backslash GQ_1))$ is an isomorphism. Further, as we saw just before Corollary~\ref{P}, the map $i''_{1}(\mathbb{S}^k):\mathbb{S}^k\backslash GQ_1\lhra\mathbb{S}^k$ is null homotopic. We conclude from~(\ref{eq:psi1}) that the homomorphism $\pi_{j}(i_1'(\mathbb{S}^k/G)): \pi_j((\mathbb{S}^k/G)\backslash \bar Q_1)\longrightarrow \pi_j(\mathbb{S}^k/G)$ is trivial. But if $1\leq i\leq n$, the composition of the map $i_n'(\mathbb{S}^k/G) : F_n((\mathbb{S}^k/G)\backslash \bar Q_1)\lhra\prod_1^n\mathbb{S}^k/G$ with the projection $p_{J_{i}(\mathbb{S}^k/G)}:\prod_1^n\mathbb{S}^k/G\longrightarrow \mathbb{S}^k/G$ factors through the map $i'_1( \mathbb{S}^k/G): (\mathbb{S}^k/G)\backslash \bar Q_1\lhra \mathbb{S}^k/G$, and so the homomorphism $\pi_j(i_n'(\mathbb{S}^k/G)) : \pi_j(F_n((\mathbb{S}^k/G)\backslash \bar Q_1))\longrightarrow\pi_j(\prod_1^n\mathbb{S}^k/G)$ is trivial.
\end{proof}

Let $\Delta_j^n(\mathbb{S}^k/G)$ denote the diagonal subgroup of $\prod_1^n\pi_j(\mathbb{S}^k/G)$. We end this paper with the following proposition.

\begin{proX}\label{diag}
Let $k\geq 3$ be odd, let $j\geq 2$, and let $G\times \mathbb{S}^k\longrightarrow\mathbb{S}^k$ be a free action of a finite group $G$ on $\mathbb{S}^k$.
Then:
\begin{enumerate}[(1)]
\item\label{it:diag1} the image of the homomorphism $\pi_j(i_n({\mathbb{S}^k/G})):\pi_j(F_n(\mathbb{S}^k/G))
\longrightarrow \prod_1^n\pi_j(\mathbb{S}^k/G)$ is the diagonal subgroup $\Delta_j^n(\mathbb{S}^k/G)$.

\item\label{it:diag2} there are split short exact sequences:
\begin{gather}
0\longrightarrow \Delta_j^n(\mathbb{S}^k/G) \longrightarrow \pi_j\biggl(\prod_1^n\mathbb{S}^k/G\biggr)\longrightarrow \pi_{j-1}\biggl(\prod_1^{n-1}\Omega(\mathbb{S}^k/G)\biggr)\longrightarrow 0 \quad\text{and}\label{eq:ses1}\\
0\longrightarrow  \pi_j(F_{n-1}^G(\mathbb{S}^k\backslash GQ_1)) \longrightarrow  \pi_j(F_n(\mathbb{S}^k/G)) \longrightarrow \Delta_j^n(\mathbb{S}^k/G) \longrightarrow 0.\label{eq:ses2}
\end{gather}
In particular, if $G$ is the trivial group then there is a split short exact sequence:
$$0\longrightarrow  \pi_j(F_{n-1}(\mathbb{R}^k)) \longrightarrow  \pi_j(F_n(\mathbb{S}^k)) \longrightarrow \Delta_j^n(\mathbb{S}^k) \longrightarrow 0.$$
\end{enumerate}
\end{proX}

\begin{proof} Let $j\geq 2$.
\begin{enumerate}[(1)]
\item First, note that the quotient map $q(\mathbb{S}^k) :\mathbb{S}^k\longrightarrow\mathbb{S}^k/G$
induces an isomorphism $\pi_j(q(\mathbb{S}^k)) :\pi_j(\mathbb{S}^k)\longrightarrow\pi_j(\mathbb{S}^k/G)$. Let $s_{j}: \pi_{j}(\mathbb{S}^k/G) \longrightarrow \pi_{j}(F_n(\mathbb{S}^k/G))$ be defined by $s_{j}= \pi_j(\psi_n(\mathbb{S}^k) \circ \lambda_i(\mathbb{S}^k)) \circ (\pi_{j}(q(\mathbb{S}^k)))^{-1}$. It follows from Lemma~\ref{fad} that:
\begin{equation*}
\pi_j(p_{J_i}(\mathbb{S}^k/G))\left\lvert_{\pi_j(F_n(\mathbb{S}^k/G))}\right. \circ s_{j}= \mbox{id}_{\pi_j(\mathbb{S}^k/G)}
\end{equation*}
for all $1\le i\le n$. In particular, the homomorphism $\pi_j(p_{J_i}(\mathbb{S}^k/G))\left\lvert_{\pi_j(F_n(\mathbb{S}^k/G))}\right.$ is surjective and admits a section $s_{j}$. Taking
the long exact sequence in homotopy of the fibration
$$F_{n-1}(\mathbb{S}^k/G\backslash \bar Q_1)\stackrel{j_{n}'(\mathbb{S}^k/G)}{\lhra} F_n(\mathbb{S}^k/G)\xrightarrow{p_{J_i}(\mathbb{S}^k/G)\left\lvert_{F_n(\mathbb{S}^k/G)}\right.} \mathbb{S}^k/G,$$
we obtain the following split short exact sequence:
\begin{equation*}
0 \longrightarrow \pi_{j}(F_{n-1}(\mathbb{S}^k/G\backslash \bar Q_1)) \xrightarrow{\pi_{j}(j_{n}'(\mathbb{S}^k/G))} \pi_{j}(F_n(\mathbb{S}^k/G))
 \xrightarrow{\pi_{j}(p_{J_i}(\mathbb{S}^k/G)\left\lvert_{F_n(\mathbb{S}^k/G)}\right.)} \pi_{j}(\mathbb{S}^k/G) \longrightarrow 0.
\end{equation*}
In particular,
\begin{equation}\label{eq:dirsum}
\pi_j(F_n(\mathbb{S}^k/G))= \operatorname{\text{Im}}(\pi_{j}(j_{n}'(\mathbb{S}^k/G))) \oplus \operatorname{\text{Im}}(s_{j}).
\end{equation}

Now let $1\leq i,l\leq n$. By Lemma~\ref{fad}, the composition
$p_{J_l}(\mathbb{S}^k/G)\left\lvert_{F_n(\mathbb{S}^k/G)}\right. \circ \psi_n(\mathbb{S}^k) \circ \lambda_i(\mathbb{S}^k)$ is homotopic to $q(\mathbb{S}^k)$, from which
we deduce that:
\begin{equation*}
\pi_j(i_n(\mathbb{S}^k/G)\circ\psi_n(\mathbb{S}^k)\circ\lambda_i(\mathbb{S}^k))(\pi_j(\mathbb{S}^k))\subseteq \Delta_j^n(\mathbb{S}^k/G).
\end{equation*}
Conversely, if $\alpha\in \pi_{j}(\mathbb{S}^k/G)$ then:
\begin{equation}\label{eq:insj}
\pi_j(i_n(\mathbb{S}^k/G)\circ\psi_n(\mathbb{S}^k)\circ\lambda_i(\mathbb{S}^k))((\pi_{j}(q(\mathbb{S}^k)))^{-1}(\alpha))=(\alpha,\ldots,\alpha),
\end{equation}
and we conclude that:
\begin{align*}
\Delta_j^n(\mathbb{S}^k/G)&= \pi_j(i_n(\mathbb{S}^k/G)\circ\psi_n(\mathbb{S}^k)\circ\lambda_i(\mathbb{S}^k))(\pi_j(\mathbb{S}^k))\\
&= \pi_j(i_n(\mathbb{S}^k/G)\circ\psi_n(\mathbb{S}^k)\circ\lambda_i(\mathbb{S}^k))\circ (\pi_{j}(q(\mathbb{S}^k)))^{-1}(\pi_j(\mathbb{S}^k/G))\\
&= \pi_j(i_n(\mathbb{S}^k/G)) \circ s_{j} (\pi_j(\mathbb{S}^k/G)),
\end{align*}
so
\begin{equation}\label{eq:imsj}
\pi_j(i_n(\mathbb{S}^k/G))(\operatorname{\text{Im}}(s_{j}))=\Delta_j^n(\mathbb{S}^k/G).
\end{equation}

In the commutative diagram of fibrations~(\ref{E'}), replace $X$ by $\mathbb{S}^k/G$, and $Q_{1}$ by $\bar{Q}_{1}$.
By Lemma~\ref{Lem}, the homomorphism $$\pi_j(i_{n-1}'(\mathbb{S}^k/G)) : \pi_j(F_{n-1}((\mathbb{S}^k/G)\backslash \bar Q_1))\longrightarrow\pi_j\biggl(\prod_1^{n-1}\mathbb{S}^k/G\biggr)$$
is trivial. Taking the long exact sequence in homotopy of~(\ref{E'}), and using the commutativity of the resulting diagram, we deduce that $\operatorname{\text{Im}}(\pi_{j}(j_{n}'(\mathbb{S}^k/G))) \subset \ker{\pi_j(i_n(\mathbb{S}^k/G))}$.
By taking the image of~(\ref{eq:dirsum}), it follows from this inclusion and~(\ref{eq:imsj}) that $\operatorname{\text{Im}}(\pi_j(i_n(\mathbb{S}^k/G)))= \Delta_j^n(\mathbb{S}^k/G)$ as required.

\item Consider the homotopy fibration
\begin{equation}\label{eq:homofinin}
I_{i_{n}(\mathbb{S}^k/G)} \xrightarrow{p_{i_{n}(\mathbb{S}^k/G)}} F_{n}(\mathbb{S}^k/G) \xrightarrow{i_{n}(\mathbb{S}^k/G)} \prod_1^{n}\mathbb{S}^k/G.
\end{equation}
Taking the long exact sequence in homotopy and using part~(\ref{it:diag1}) and the isomorphism between $\pi_{j}(I_{i_{n}(\mathbb{S}^k/G)})$ and $\pi_{j}(F_{n-1}^G(\mathbb{S}^k\backslash GQ_1))$ given by Proposition~\ref{AS},
we obtain the short exact sequence~(\ref{eq:ses2}). Let $s_{j}': \Delta_j^n(\mathbb{S}^k/G) \longrightarrow \pi_{j}(F_{n}(\mathbb{S}^k/G))$ be defined by $s_{j}'=s_{j}\circ \pi_j(p_{J_i}(\mathbb{S}^k/G))\left\lvert_{\Delta_j^n(\mathbb{S}^k/G)}\right.$. If $(\alpha,\ldots, \alpha)\in \Delta_j^n(\mathbb{S}^k/G)$ then $\pi_{j}(i_{n}(\mathbb{S}^k/G)) \circ s_{j}'(\alpha,\ldots, \alpha)=\pi_{j}(i_{n}(\mathbb{S}^k/G)) \circ s_{j}(\alpha)=(\alpha,\ldots, \alpha)$ by~(\ref{eq:insj}), which shows that~(\ref{eq:ses2}) splits.

We now derive the short exact sequence~(\ref{eq:ses1}). We take $X=\mathbb{S}^k$ in Theorem~(\ref{l})(\ref{it:l3}). The hypotheses of the theorem are satisfied because the inclusion $\mathbb{S}^k \backslash GQ_1 \lhra \mathbb{S}^k$ is null homotopic, and the homomorphism $\pi_{m}(q(\mathbb{S}^k)): \pi_{m}(\mathbb{S}^k) \longrightarrow \pi_{m}(\mathbb{S}^k/G)$ is an isomorphism for all $m\geq 2$ and is injective if $m=1$. As in the statement of Theorem~(\ref{l})(\ref{it:l3}), we consider the homotopy fibration~(\ref{eq:homofinin}), and we identify $\pi_{j-1}(I_{i_{n}(\mathbb{S}^k/G)})$ with $\pi_{j}(F_{n-1}^G(\mathbb{S}^k\backslash GQ_1))\times \pi_{j-1}(\prod_1^{n-1}\Omega(\mathbb{S}^k/G))$, and $\pi_{j-1}(\Omega(\mathbb{S}^k/G))$ with $\pi_{j}(\mathbb{S}^k/G)$. It then follows that the restriction to the subgroup $\pi_{j}(\prod_1^{n-1}\mathbb{S}^k/G)$ of $\pi_{j}(\prod_1^{n}\mathbb{S}^k/G)$ of the boundary homomorphism $\hat\partial_{j}: \pi_{j-1}(\Omega(\prod_1^n \mathbb{S}^k/G)) \longrightarrow \pi_{j-1}(I_{i_n(\mathbb{S}^k/G)})$ coincides with the inclusion of $\pi_{j}(\prod_1^{n-1}\mathbb{S}^k/G)$ in $\pi_{j-1}(F^G_{n-1}(\mathbb{S}^k\backslash GQ_1))\times\pi_{j-1}(\Omega(\prod_1^{n-1}\mathbb{S}^k/G))$. On the other hand, if $\alpha\in \pi_{j}(\mathbb{S}^k/G)$ then by part~(\ref{it:diag1}), $(\alpha,\ldots,\alpha)\in \operatorname{\text{Im}}(\pi_{j}(i_{n}(\mathbb{S}^k/G)))$, and so $\hat\partial_{j}(\alpha,\ldots,\alpha)=0$ by exactness. Applying Theorem~(\ref{l})(\ref{it:l3}) once more, we see that $\hat\partial_j(0,\ldots,0,\alpha)=\hat\partial_j(-\alpha,\ldots,-\alpha,0)=(0, -\alpha,\ldots,-\alpha)$, where in the final expression, $0$ (resp.\ $(-\alpha,\ldots,-\alpha)$) corresponds to the $\pi_{j}(F_{n-1}^G(\mathbb{S}^k\backslash GQ_1))$-factor (resp.\ to the $\pi_{j-1}(\prod_1^{n-1}\Omega(\mathbb{S}^k/G))$-factor). It follows from these considerations that $\operatorname{\text{Im}}(\hat\partial_j)=\pi_{j-1}(\prod_1^{n-1}\Omega(\mathbb{S}^k/G))$. We then obtain the short exact sequence~(\ref{eq:ses1}) by noting that $\ker(\hat\partial_j)=\Delta_j^n(\mathbb{S}^k/G)$ by exactness and part~(\ref{it:diag1}). The fact that~(\ref{eq:ses1}) splits is a consequence of Theorem~(\ref{l})(\ref{it:l3}).\qedhere
\end{enumerate}
\end{proof}

Note that Proposition \ref{diag} gives a complete description of the long exact homotopy sequence in homotopy of the homotopy fibration~(\ref{eq:homofinin}) in the case where $G$ is finite and $k$ is odd. The study of the configuration spaces $F_n(\mathbb{R}P^{2k})$ constitutes work in progress.

\vspace{3mm}

\noindent{\bf Faculty of Mathematics and Computer Science\\
University of Warmia and Mazury\\
S\l oneczna 54 Street\\
10-710 Olsztyn, Poland}\\
e-mail: \texttt{marekg@matman.uwm.edu.pl}

\vspace{1.8mm}

\noindent{\bf Department of Mathematics-IME\\
University of S\~ao Paulo\\
Caixa Postal 66.281-AG. Cidade de S\~ao Paulo\\
05314-970 S\~ao Paulo, Brasil}\\
e-mail: \texttt{dlgoncal@ime.usp.br}

\vspace{1.8mm}

\noindent{\bf Normandie Univ., UNICAEN, CNRS,\\
Laboratoire de Math\'ematiques Nicolas Oresme UMR CNRS~\textup{6139},\\
14000 Caen, France.}\\
e-mail: \texttt{john.guaschi@unicaen.fr}

\end{document}